\RequirePackage[l2tabu, orthodox]{nag}

\documentclass[12pt]{amsart}

\usepackage{fullpage,url,amssymb,enumitem,colonequals}
\setenumerate{listparindent=\parindent}
\setlist[enumerate]{label={\upshape(\arabic*)}}
\usepackage{tabularx} 
\usepackage{booktabs} 
\usepackage{ragged2e}  
\newcolumntype{Y}{>{\RaggedRight\arraybackslash}X} 

\usepackage{mathrsfs} 

\usepackage[OT2,T1]{fontenc}
\usepackage{mathtools}
\usepackage{upgreek} 
\DeclareSymbolFont{cyrletters}{OT2}{wncyr}{m}{n}
\DeclareMathSymbol{\Sha}{\mathalpha}{cyrletters}{"58}

\usepackage{color}

\newcommand{\Aff}{\mathbb{A}}

\newcommand{\F}{\mathbb{F}}

\newcommand{\PP}{\mathbb{P}}

\newcommand{\Klaur}[1]{{K(\!(#1)\!)} }
\newcommand{\Kpower}[1]{{K[\![#1]\!]} }

\newcommand{\calL}{\mathcal{L}}

\newcommand{\calO}{\mathcal{O}}

\newcommand*{\TakeFourierOrnament}[1]{{%
		\fontencoding{U}\fontfamily{futs}\selectfont\char#1}}
\newcommand*{\danger}{\TakeFourierOrnament{66}} 

\DeclareMathOperator{\Aut}{Aut}

\DeclareMathOperator{\Char}{char}

\DeclareMathOperator{\Gal}{Gal}

\DeclareMathOperator{\im}{im}

\DeclareMathOperator{\Sym}{Sym}

\newcommand{\et}{{\text{\'et}}}

\newcommand{\AGL}{\operatorname{AGL}}
\newcommand{\AGaL}{\operatorname{A \Gamma L}}

\newcommand{\GL}{\operatorname{GL}}

\newcommand{\PGL}{\operatorname{PGL}}
\newcommand{\PSL}{\operatorname{PSL}}
\newcommand{\PGaL}{\operatorname{P \Gamma L}}

\newcommand{\slantsf}[1]{\textsl{\textsf{#1}}}

\newtheorem{theorem}{Theorem}[section]
\newtheorem{lemma}[theorem]{Lemma}

\newtheorem{proposition}[theorem]{Proposition}

\theoremstyle{definition}
\newtheorem{definition}[theorem]{Definition}

\newtheorem{example}[theorem]{Example}

\theoremstyle{remark}
\newtheorem{remark}[theorem]{Remark}

\usepackage{microtype}

\usepackage[
	backref,
	pdfauthor={Borys Kadets}, 
]{hyperref}
\usepackage[alphabetic,backrefs,lite]{amsrefs} 
\makeatletter
\@namedef{subjclassname@2020}{%
	\textup{2020} Mathematics Subject Classification}
\makeatother
\begin{document}

\title[sectional monodromy groups]{Sectional monodromy groups of projective curves}

\author{Borys Kadets}
\thanks{This research was supported in part by Simons Foundation grant \#402472 (to Bjorn Poonen)}
\address{Department of Mathematics, Massachusetts Institute of Technology, Cambridge, MA 02139-4307, USA}
\email{bkadets@mit.edu}
\urladdr{\url{http://math.mit.edu/~bkadets/}}

\begin{abstract}
Fix a degree $d$ projective curve $X \subset \PP^r$ over an algebraically closed field $K$. Let $U \subset (\PP^r)^*$ be a dense open subvariety  such that every hyperplane $H \in U$ intersects $X$ in $d$ smooth points. Varying $H \in U$ produces the monodromy action $\varphi: \pi_1^{\et}(U) \to S_d$. Let $G_X\colonequals \im(\varphi)$. The permutation group $G_X$ is called the sectional monodromy group of $X$. In characteristic zero $G_X$ is always the full symmetric group, but sectional monodromy groups in characteristic $p$ can be smaller. For a large class of space curves ($r \geqslant 3$) we classify all possibilities for the sectional monodromy group $G$ as well as the curves with $G_X=G$. We apply similar methods to study a particular family of rational curves in $\PP^2$, which enables us to answer an old question about Galois groups of generic trinomials.
\end{abstract}

\maketitle

\section{Introduction}

Let $K$ be an algebraically closed field. Consider an integral nondegenerate proper curve $X$ in the projective space $\PP^r$ over $K$.
Choose coordinates $x_0, ..., x_r$ on $\PP^r$. Extend scalars from $K$ to $L=K(t_1, ...,t_r)$ and consider the subscheme
$Z \subset \PP^r_{L}$ given by the intersection of $X_{L}$ with the hyperplane $x_0 + t_1x_1+ ... + t_rx_r=0$. The scheme $Z$ is the spectrum of a finite separable extension $M$ of $L$ (see Lemma~\ref{irreducible}). The Galois group $G$ of the Galois closure of $M$ over $L$ is our main object of interest and will be referred to as the \slantsf{sectional monodromy group}. One can think of $G$ geometrically as the monodromy of the hyperplane section $H \cap X$ as the hyperplane $H$ varies.  A folklore result is the following.

\begin{proposition}[see Remark~\ref{introzero}]\label{boring}
Assume that the characteristic of $K$ is zero. Let $d$ be the degree of $X$. Then $G \cong S_d$.
\end{proposition}

The result is important in the study of curves in projective space. For instance, it can be used to restrict the possibilities for Hilbert polynomials of space curves; for more details see \cite{Harris1980}, \cite{Eisenbud-Harris1982}.

The $\Char K = 0$ assumption of Proposition~\ref{boring} cannot be removed. For example, when $\Char K = p$, the sectional monodromy group of the rational plane curve $x^p=yz^{p-1}$ is $\AGL_1(p)$. 

Motivated by geometric applications, Rathmann~\cite{Rathmann1987} proved a version of Proposition~\ref{boring} valid in all characteristics.

\begin{theorem}[Rathmann]\label{Rathmann}
	Assume that $X \subset \PP^r$ is a smooth nondegenerate curve of degree $d$. Assume $r \geqslant 4$. Then the sectional monodromy group of $X$ contains the alternating group $A_d$.
\end{theorem}

\begin{remark}
	Gareth Jones noticed that Rathmann's paper \cite{Rathmann1987} contains a group-theoretic error: the list of triply transitive groups in \cite{Rathmann1987}*{Theorem~2.4} is incomplete. It is unclear if this gap is easy to fix. For instance, one of the omitted Mathieu groups $\Aut(M_{22})$ is not contained in the alternating group, so the proof of \cite{Rathmann1987}*{Proposition~2.14} requires a separate analysis of this case.
\end{remark}
Rathmann also gave an example of a smooth rational curve in $\PP^3$ with $G=\PGL_2(q)$. Rathmann's theorem is important for understanding statistics over function fields, for example the large characteristic version of the Bateman-Horn conjecture (see \cite{Entin-preprint}*{Section~6} for examples of such statistical problems). In Section~\ref{tangents} we prove an extension of Rathmann's theorem that applies to nonsmooth curves in $\PP^3$. Recall that a projective curve is called \slantsf{strange} if all of its tangent lines pass through a common point. Strange curves are rare, for example the only smooth strange curves are lines in any characteristic and conics in characteristic $2$. The tangent variety of $X$ (also called the tangent developable surface) is the Zariski closure of the union of lines tangent to $X$. 

\begin{theorem}\label{intro-sectional}
	 Assume $X \subset \PP^r$ is an integral nondegenerate curve of degree $d$. Assume that $r \geqslant 3$ and $X$ is not strange. Then exactly one of the following is true. \hfill 
	\begin{enumerate}
		\item The sectional monodromy group of $X$ contains $A_d$.
		
		\item $r=3$ and the curve $X$ is projectively equivalent to the projective closure of the rational curve $\{(t,t^q, t^{q+1}), t \in K\} \subset \Aff_K^3$ for some power $q$ of the characteristic. The sectional monodromy group of $X$ is $\PGL_2(q)$.
		
		\item $r=3$, $\Char K =2$, the tangent variety of $X$ is a smooth quadric and $G \subset \AGL_n(2)$.
	\end{enumerate}

\end{theorem}

Note that the curves that satisfy the assumptions of Theorem~\ref{intro-sectional} and have small sectional monodromy groups are either rational or contained in a smooth quadric. Rational curves and curves lying on quadrics are not important for the applications to the degree-genus problem.

For plane curves, sectional monodromy groups are not known even in the following very simple case.

\begin{example}
	Let $X_{n,m}$ be the rational curve $x^n=y^mz^{n-m}$ for some relatively prime integers $n>m>0$. Then the sectional monodromy group of $X$ is the Galois group of the trinomial $x^n + ax^m +b$ over $K(a,b)$. Classifying possible Galois groups of generic trinomials is an old  problem; see \cite{Cohen1980}, \cite{Smith1977}, \cite{Uchida1970}, \cite{Smith1984}. A closely related problem of classifying Galois groups of trinomials over the function field $K(t)$ when $a$ and $b$ are specialized to powers of $t$ has also been widely studied; see \cite{Abhyankar1992}, \cite{Abhyankar1993}, \cite{Abhyankar1993.2}, \cite{Abhyankar1994}, \cite{Abhyankar-Seiler-Popp1992}, \cite{Conway-et-al2010}.
\end{example}

The curves $X_{m,n}$ do not satisfy both of the assumptions of Theorem~\ref{intro-sectional}: they are evidently plane curves ($r=2$) and many of them are strange. Nevertheless, we can use methods similar to those in the proof of Theorem~\ref{intro-sectional} to study the curves $X_{m,n}$. In section~\ref{trinomials} we prove the following theorem.
\begin{theorem}\label{intro-trinomials}
	Let $K$ be an algebraically closed field of characteristic $p$. Let $n,m$ be a pair of relatively prime integers with $m<n$. Let $G$ be the Galois group of the polynomial $x^n+ax^m+b$ over $K(a,b)$.
	\begin{enumerate}
		\item If $n=p^d$ and $m=1$ or $m=p^d-1$, then $G=\AGL_1(p^d)$.
		\item If $n=6$, $m=1$ or $m=5$, and $p=2$, then $G=\PSL_2(5)$.
		\item If  $n=12$, $m=1$ or $m=11$, and $p=3$, then $G=M_{11}$ (a Mathieu group, see \cite{Conway1999}).
		\item If $n=24$, $m=1$ or $m=23$, and $p=2$, then $G=M_{24}$.
		\item If  $n=11$, $m=2$ or $m=9$, and $p=3$, then $G=M_{11}$.
		\item If $n=23$, $m=3$ or $m=20$, and $p=2$, then $G=M_{23}$.
		\item Let $q\colonequals p^k$ for some integer $k$. If $n=(q^d-1)/(q-1)$, $m=(q^s-1)/(q-1)$ or $m=n-(q^s-1)/(q-1)$ for some  integers $s,d$, then $G=\PGL_d(q)$.
		\item If none of the above holds, then $A_n \subset G$.	
	\end{enumerate}
\end{theorem}
{\danger \bf Warning.} The aforementioned result of Rathmann as well as the results of this paper depend on the classification of finite simple groups.
\section{Transitivity of Sectional Monodromy Groups}

From now on, let $K$ denote a fixed algebraically closed field. We first give a more geometric definition of the sectional monodromy groups from the introduction. Let $r \geqslant 2$ be an integer and let $X \subset \PP^r$ be a degree $d$ projective nondegenerate integral curve. Consider the incidence scheme $I=I_X \subset X \times (\PP^r)^*$ parameterizing pairs $(P, H)$ of points $P \in X$ and hyperplanes $H \subset \PP^r$ subject to the condition $P \in H$. The scheme $Z$ from the introduction is the generic fiber of the projection $\pi=\pi_X\colon I_X \to (\PP^r)^*.$ We call the covering $\pi_X$ the sectional covering associated to $X$. 

\begin{lemma}\label{irreducible}
	The incidence scheme $I$ is irreducible. The morphism $\pi$ is finite of degree $d$ and generically \'{e}tale.
\end{lemma}

\begin{proof}
	The fibers of the first coordinate projection $I \to X$ are projective spaces of dimension $r-1$; hence $I$ is irreducible. The morphism is generically \'{e}tale by Bertini's theorem \cite{Jouanolou1983}*{Proposition~6.4}, and finite since it is quasi-finite and proper.
\end{proof}

Lemma~\ref{irreducible} implies that the sectional monodromy group $G$ is well-defined and acts transitively on a set of $d$ points. If $U \subset (\PP^r)^*$ is a dense open subset over which $\pi$ is \'{e}tale and $\Omega$ is a fiber of $\pi$ over a geometric point of $U$, then $G$ can be identified with the image of $\pi_1^{\et}(U)$ in $\mathrm{Sym}(\Omega)$. If $Z \subset U$ is a closed subset, then the monodromy group of $\pi^{-1}(Z) \to Z$ is a permutation subgroup of $G$. We now show that the action of $G$ on $\Omega$ is always 2-transitive, and has even higher transitivity degree if the curve is not ``strange''.

\begin{lemma}\label{two-trans}
	The sectional monodromy group acts doubly transitively.
\end{lemma}
\begin{proof}
	Choose a hyperplane $H$ that intersects $X$ at $d$ distinct smooth points. Let $P$ be one of the points in $H \cap X$. Choose an $(r-2)$-dimensional subspace $V$ in $\PP^r$ such that $V \cap X = P$. 
	Consider the line $l \subset (\PP^r)^*$ corresponding to the family of hyperplanes containing $V$. Let $Z=\pi^{-1}(l)$. Projection of $Z$ to $X$ is an isomorphism over $X \setminus P$. The fiber above $P$ is isomorphic to $l$ (via $\pi$). Therefore $Z$ decomposes into a union of two irreducible components $Z=Z_1 \cup Z_2$, where $\pi\colon Z_1 \to l$ is an isomorphism and $Z_2$ is isomorphic to $X$. Since the fiber of $\pi'=\pi|_{Z_2}\colon Z_2 \to l$ above the point of $l$ corresponding to $H$ is smooth, the covering $\pi'$ is generically \'{e}tale. Therefore, the monodromy group $G$ contains a subgroup that acts via fixing one point and acting transitively on the others. Thus $G$ is doubly transitive.
\end{proof}

For every $r$ there are examples of integral nondegenerate curves $X \subset \PP^r$ for which $G_X$ is not triply transitive (see \cite{Rathmann1987}*{Example~1.2}). However we will prove that $G_X$ is $r$-transitive if $X$ is not ``strange''. We recall the definition of a strange curve.

\begin{definition}
	An integral curve $X$ in $\PP^r$ is called \slantsf{strange} if the tangent lines to $X$ at the smooth points pass through a common point.
\end{definition}

\begin{remark}
	Equivalently, an integral curve is strange if the tangent lines to $X$ at all but finitely many smooth points pass through a common point.
\end{remark}
A secant of a curve $X$ is a line intersecting it in at least $2$ points.
\begin{proposition}\label{multi-secant}
	If every secant intersects $X$ in at least three points, then $X$ is strange.
\end{proposition}
\begin{proof}
	See \cite{Hartshorne1977}*{Proposition~3.3.8}.
\end{proof}

Most of the results that follow concern curves that are not strange.
It is natural to expect higher transitivity for nonstrange curves. This is the case, as we show in Proposition~\ref{r-transitive}. We precede the proof of Proposition~\ref{r-transitive} with two lemmas.

\begin{lemma}\label{threadthetangents}
	Let $X \subset \PP^r$, $r \geqslant 3$ be a proper nondegenerate nonstrange integral curve. Let $\calL$ denote the set of lines in $\PP^r$ that have nonempty intersection with every tangent to $X$. Then either $\calL$ is finite or $r=3$, $X$ is contained in a smooth quadric, and $\calL$ is one of the quadric's rulings.
\end{lemma}
\begin{proof}
	Assume that $\calL$ is infinite. Consider two distinct lines $\ell_1, \ell_2 \in \calL$. If $\ell_1$ and $\ell_2$ intersect, then either every tangent to $X$ passes through $\ell_1 \cap \ell_2$ or every tangent to $X$ is in the plane containing $\ell_1$ and $\ell_2$. Since $X$ is nondegenerate and nonstrange, neither of these is possible. Therefore any two distinct lines from $\calL$ are skew. Consider two skew lines $\ell_1, \ell_2 \in \calL$. Since every tangent of $X$ crosses both of them, the curve $X$ lies in the dimension three projective subspace of $\PP^r$ that contains $\ell_1, \ell_2$. Hence $r=3$. A classical fact, which we prove below, is that the union of lines that intersect three pairwise skew lines is a smooth quadric. Thus there exists a smooth quadric $Q$ containing every tangent line of $X$. Elements of $\calL$ are pairwise skew and lie on $Q$, therefore $\calL$ forms one of the rulings of $Q$.
	
	We now show that the union of lines intersecting three fixed lines in $\PP^3$ is a smooth quadric. Since $10=\dim H^0(\PP^3, \calO(2)) > 3 \dim H^0(\PP^1, \calO(2))=9$, there exists a quadric containing any triple of lines. If the lines are pairwise skew, the quadric cannot be a cone or a pair of planes; therefore it must be smooth. Every line intersecting a smooth quadric in at least three points has to lie on it.
\end{proof}

\begin{lemma}\label{project}
	Let $X \subset \PP^r$, $r \geqslant 3$ be a nonstrange nondegenerate integral curve. If $r=3$ assume additionaly that $X$ is not contained in a quadric. Then the projection from a general point $P$ of $X$ is birational onto its image. Moreover, the image of a general projection is nonstrange.
\end{lemma}
\begin{proof}
   Consider the set $\calL$ of Lemma~\ref{threadthetangents}. Since $X$ is not contained in a quadric, $\calL$ is finite.
	Proposition~\ref{multi-secant} implies that a general point of $X$ lies on finitely many multisecants. Therefore projection from a general point is generically one-to-one on $K$ points and is thus birational. Let $P \in X(K)$ be a point such that the projection  $\pi_P \colon X \to Y$ from $P$ is birational onto its image $Y$, and such that $P$ does not lie on a line from $\calL$. Suppose that $Y$ is strange and all the tangents pass through $Q \in \PP^{r-1}$. Then the line $\pi_P^{-1}(Q) \subset \PP^r$ is an element of $\calL$ that contains $P$. Thus $Y$ is nonstrange. 

\end{proof}

\begin{proposition}\label{r-transitive}
	If $X \subset \PP^r$, $r \geqslant 2$ is an integral nondegenerate nonstrange curve, then $G$ is $r$-transitive.
\end{proposition}
\begin{proof}
	We prove the statement by induction on $r$. The case $r=2$ is Lemma~\ref{two-trans}. Assume $r>2$. Assume first that it is not the case that $r=3$ and $X$ lies on a quadric. Choose a point $P \in X(K)$ such that the conclusions of Lemma~\ref{project} hold. Let $Y \subset \PP^{r-1}$ be the projection of $X$ from $P$. Consider the hyperplane $H \in (\PP^r)^*$ corresponding to the family of hyperplanes in $\PP^r$ passing through $P$.  Let $Z$ denote $\pi^{-1}(H)$. The scheme $Z$ has an irreducible component $Z'=P \times H \subset X \times (\PP^r)^*$. The covering $Z \setminus Z' \to H$ is birational to the sectional covering of $Y$, since the preimages of hyperplanes  in $\PP^{r-1}$ are hyperplanes in $\PP^r$ containing $P$. By the induction hypothesis the monodromy group of the covering $Z \setminus Z' \to H$ is $r-1$-transitive. The covering $Z' \to H$ is an isomorphism. Therefore the monodromy group of $Z \to H$ acts by fixing one point and acting $(r-1)$-transitively on the rest. Since $G$ acts transitively  and contains a subgroup acting $(r-1)$-transitively on all but one (fixed) point, $G$ is $r$-transitive.
	
	We have to show that when $r=3$ and $X$ lies on a quadric $Q$ the group is $3$-transitive. Consider a generic plane $H$. We can assume that both $H \cap X$ and $H \cap Q$ are smooth.   Take two points of $H \cap X$ and a line $\ell$ through them. The line $\ell$ contains exactly two points $P_1, P_2$ of the curve. Consider the line $\ell^* \in (\PP^3)^*$ corresponding to the hyperplanes containing $\ell$. The restriction of the sectional covering of $X$ to $\ell^*$ is generically \'{e}tale since $H \cap X$ is smooth. Let $Y \subset X \times \ell^*$ denote the preimage of $\ell$ under $\pi_X$. The projection of $Y$ to $X$ is an isomorphism over $X\setminus \{P_1, P_2\}$. The preimages of $P_1$, $P_2$ in $Y$ are each isomorphic to $\ell^*$. Therefore $Y$ decomposes into a union of irreducible components $Y=Y_1 \bigcup Y_2 \bigcup Y_3$, where $Y_1$ is isomorphic to $X$ and $Y_2, Y_3$ are isomorphic to $\ell^*$. Hence the monodromy group contains a subgroup that acts by fixing two points and acting transitively on the rest. Therefore $G$ is triply transitive.
\end{proof}

Multiple transitivity is a very strong condition on the group. For example, the only $5$-transitive groups are $A_n$ and $S_n$ in their natural actions and the Mathieu groups $M_{12}$ and $M_{24}$ acting on $12$ and $24$ points respectively. The classification of $2$-transitive groups is known; see, for example, \cite{Dixon-Mortimer1996}*{\S 7.7} for the statement.

\begin{remark}
	It is common to state classification results on permutation groups without specifying the action (or by referring to the ``natural'' action). This is relatively harmless except for the cases of $M_{11}$, which has two multiply-transitive actions on $11$ and $12$ points, and of $M_{12}$, which has two multiply transitive actions on $12$ points exchanged by an outer automorphism.
\end{remark}

The following proposition due to Rathmann describes the behavior of the sectional monodromy group with respect to projections from points outside of the curve.

\begin{proposition}[\cite{Rathmann1987}*{Proposition~1.10.}]\label{G-under-proj}
Let $X \subset \PP^r$, $r \geqslant 3$ be a projective nondegenerate integral curve. Let $Y$ be the projection of $X$ from a general point of $\PP^r$. Then the sectional monodromy groups of $X$ and $Y$ are isomorphic.	
\end{proposition}

\section{Tangents and inertia groups}\label{tangents}

We will compute sectional monodromy groups by restricting the sectional monodromy covering $\pi_X$ to a line in $(\PP^r)^*$ and computing inertia groups. A natural choice of a point in $(\PP^r)^*$ above which the restricted covering ramifies is a point corresponding to a hyperplane containing a tangent line of $X$. Tangent hyperplanes can behave pathologically in characteristic $p$ (even for nonstrange curves).

For the statements and precise definitions in this paragraph see \cite{Hefez1989} and references therein. Let $X \subset \PP^r$ be an integral projective curve that is not a line. The dual curve $X^* \subset \mathrm{Gr}(2,r+1)$ is defined to be the closure of the set of tangent lines to $X$. The Gauss map is the rational map $\upgamma: X \dashrightarrow X^*$ that sends a point to the tangent line at that point. In characteristic $0$ the Gauss map is birational. In characteristic $p$ both the separable degree $s$ and the inseparable degree $q$ of $\upgamma$ can be greater than one. A general tangent line of $X$ has $s$ points of tangency with multiplicity $\max(q,2)$ at each one. It is a theorem of Kaji that Gauss maps can be arbitrary bad in the sense that any finite inseparable extension of function fields can be realized as the generic fiber of a Gauss map; see \cite{Kaji1989}*{Corollary~3.4}. 

We calculate inertia groups by computing Galois groups of certain extensions of $\Klaur{t}$ given by adjoining to $\Klaur{t}$ zeros of convergent power series of a specific shape. Our main tool is the theory of Newton polygons (see \cite{Koblitz1984}*{IV.4}) that determines the valuations of the roots of a power series from the valuations of its coefficients. We combine the information obtained from Newton polygons with basic ramification theory to determine the local Galois groups; for an introduction to ramification theory see \cite{SerreLocalFields1979}*{Chapters\ I\ndash III}.

We will use the following group theoretic result due to Jones \cite{Jones2014}. Recall that a permutation group $G$ acting on $\Omega$ is primitive if it preserves no nontrivial partition of $\Omega$.

\begin{theorem}[Jones]\label{Jones}
	Let $G$ be a primitive permutation group of finite degree $n$, not containing $A_n$. Suppose that $G$ contains a cycle fixing exactly $k$ points, where $0 \leqslant k \leqslant n-2$. Then one of the following holds:
	
	\begin{enumerate}
		\item\label{cycle0} $k=0$ and either
		\begin{enumerate}
			\item $C_p \subset G \subset \AGL_1(p)$ with $n=p$ prime, or
			\item $\PGL_d(q) \subset G \subset \PGaL_d(q)$ with $n=(q^d-1)/(q-1)$ and $d \geqslant 2$ for some prime power $q$, or
			\item $G=\PSL_2(11), M_{11}$ or $M_{23}$ with $n=11, 11, 23$ respectively.
		\end{enumerate}
	
		\item $k=1$ and either
	\begin{enumerate}
		\item $\AGL_d(q) \subset G \subset \AGaL_d(q)$ with $n=q^d$ and $d \geqslant 1$ for some prime power $q$, or
		\item $G=\PSL_2(p)$ or $G=\PGL_2(p)$ with $n=p+1$ for some prime $p \geqslant 5$, or
		\item $G=M_{11}, M_{12}$ or $M_{24}$ with $n=12, 12, 24$ respectively. respectively.
	\end{enumerate}
	\item $k=2$ and $\PGL_2(q) \subset G \subset \PGaL_2(q)$ with $n=q+1$ for some prime power $q$.
	\end{enumerate}
\end{theorem}

The inclusions in Theorem~\ref{Jones} are inclusions of permutation groups. We recall the definitions of some of the groups appearing in Theorem~\ref{Jones}. The group $\AGL_d(q)$ is the affine general linear group acting on $\Aff^d(\F_q)$; it is an extension of $\GL_d(\F_q)$ by the group of translations $(\F_q)^d$. The group $\AGaL_d(q)$ is generated by $\AGL_d(q)$ and the Frobenius automorphism. The group $\PGL_d(q)=\GL_{d}(q)/\F_q^{\times}$ is the projective general linear group acting on the projective space $\PP^{d-1}(\F_q)$.  The projective semilinear group $\PGaL_d(q)$ is generated by $\PGL_d(q)$ and the Frobenius automorphism. 

In Lemma~\ref{inertia} we use inertia groups to construct special subgroups of the sectional monodromy group; together Proposition~\ref{r-transitive} and Lemma~\ref{inertia} will reduce most of the computation of sectional monodromy groups to pure group theory. First we need to prove the following lemma.

\begin{lemma}\label{AGL1}
	Let $K$ be an algebraically closed field of characteristic $p$ and let $\left(F,v\right)$ denote the valued field $\Klaur{t}$. Let $\calO$ denote its ring of integers $\calO = \Kpower{t}$. Let $P \in \calO[\![ x ]\!]$ be an arbitrary power series and let $q$ be a power of $p$. Consider a power series $Q=x^{q+2}P(x)+a_{q+1}x^{q+1}+a_qx^q+a_1x+a_0 \in \calO[\![ x ]\!]$, where $v(a_0)=v(a_1)=1$, $v(a_q)=0$. Assume  that $v(a_1a_q-a_0a_{q+1})=1$. Assume that all roots of $Q$ in the open unit disk  $D\colonequals \{x \in \overline{F}: \ v(x)>0\}$ lie in $F^{\mathrm{sep}}$, so $Q$ has exactly $q$ roots in $D$. The Galois action on these roots defines a homomorphism $\varphi \colon \Gal (\overline{F}/F) \to S_q$. Let $H$ denote the image of $\varphi$. Then $H$ equals $\AGL_1(q)$ embedded via its action on the affine line.
\end{lemma} 

\begin{proof}
	The Newton polygon of $Q$ has a unique segment of positive slope, it connects $(0,1)$ and $(q,0)$. Therefore $Q$ has $q$ roots in the open unit disk each with valuation $1/q$. Let $\alpha$ be a root of $Q$. The field extension $L=F(\alpha)/F$ has ramification index at least $q$; therefore $[L:F]=q$ and $H$ acts transitively. Consider the roots of $Q(x+\alpha)$ over $L$. We have 
	\begin{multline*}
		Q(x+\alpha)=x^q(x+\alpha)^2P(x+\alpha)+ \alpha^q (x+\alpha)^2 P(x+\alpha) + a_{q+1}x^q(x+\alpha) \\
		+ a_{q+1} \alpha^q (x+\alpha) + a_qx^q + a_1x + a_1\alpha + a_q\alpha^q+a_0.
	\end{multline*}
	The constant term must be zero since $\alpha$ is a root of $Q$. The constant term of $Q(x+\alpha)/x$ is 
	\[2\alpha^{q+1}P(\alpha) + \alpha^{q+2}P'(\alpha) + a_{q+1}\alpha^q + a_1.\]
	Assume $v(a_{q+1}\alpha^q + a_1) > 1$. Then $a_{q+1}\alpha^q+a_1 \equiv 0 \pmod{\alpha^{q+1}}$. Reducing the equation $Q(\alpha)=0$ modulo $\alpha^{q+1}$, we conclude that 
	$a_0+a_q\alpha^q \equiv 0 \pmod{\alpha^{q+1}}$. Hence $\alpha^q \equiv -a_0/a_{q} \pmod{\alpha^{q+1}}$ and $a_1-a_0a_{q+1}/a_{q} \equiv a_{q+1}\alpha^q+a_1 \equiv 0 \pmod{\alpha^{q+1}}$. But $v(a_{q+1}a_0/a_q - a_1)=1$, contradiction. Thus $v(a_{q+1}\alpha^q+a_1)=1$, and the constant term of $Q(x+\alpha)/x$ has valuation $1$.
	
	 The power series $Q(x+\alpha)/x=\vcentcolon \sum_k b_k x^k$ has constant term $b_0$ of valuation $1$. Modulo $\alpha^q$ we have $Q(x+\alpha)/x \equiv x^{q-1}(x+\alpha)^2P(x)+a_{q+1}x^{q-1}(x+\alpha)+a_qx^{q-1} \pmod {\alpha^q}$. So for $1 \leqslant k \leqslant q-2$ the valuation of $b_k$ is at least one. The valuation of $b_{q-1}$ is zero. Therefore the Newton polygon of $Q(x+\alpha)/x$ has only one segment of positive slope, it connects $(0,1)$ to $(q-1,0)$. The field $L=\Klaur{\alpha}$ has a unique extension of degree $q-1$, namely $\Klaur{\alpha^{1/(q-1)}}$. We conclude that the stabilizer of a point in  $H$ is a cyclic group of order $(q-1)$. A cycle in an imprimitive group has to either belong to a block or contain the same number of elements in every block. Since $q-1$ and $q$ are relatively prime, a $(q-1)$-cycle belongs to a single block. So the group cannot be imprimitive. Since $F$ is a nonarchimedean local field, $H$ is solvable. Theorem~\ref{Jones} implies that in this case $\AGL_1(q) \subset H \subset \AGaL_1(q)$. The group $\AGaL_1(q)$ is the group of semilinear transformation of the affine line, namely the group generated by $\AGL_1(q)$ and the Frobenius automorphism. Since the stabilizer of a point in $H$ is a cyclic group, $H=\AGL_1(q)$.  

\end{proof}

\begin{remark}
	Lemma \ref{AGL1} can be proved without using Theorem \ref{Jones} with a lengthier but elementary group theory argument. The main Theorems~\ref{intro-sectional}~and~\ref{intro-trinomials} rely on Theorem \ref{Jones} in a more substantial way.
\end{remark}

\begin{lemma}\label{inertia}
	Let $X \subset \PP^2$ be a nondegenerate nonstrange integral plane curve of degree $d$. Assume that the Gauss map of $X$ has separable degree $s$ and inseparable degree $q$. Choose a geometric point $x$ above which the covering $\pi_X$ is \'{e}tale. Denote by $\Omega$ the fiber of $\pi_X$ above $x \in (\PP^2)^*$, so the group $G_X$ acts on $\Omega$ by permutations and $\# \Omega = d$. Then the following hold:
	\begin{enumerate}
		\item If $q=1$, then $G_X$ contains a transposition.
		\item If $q>1$, then there is a decomposition $\Omega=A_1 \sqcup ... \sqcup A_s \sqcup B$, $\#A_i = q, \# B = d-sq$, a subgroup $H \subset G_X$ and a collection of subgroups $H_i \subset G_X$, $i=1,...,s$ such that the following properties hold.
		\begin{itemize}
		\item $H$ fixes every point of $B$ and preserves the decomposition
		$\Omega=A_1 \sqcup ... \sqcup A_s \sqcup B$.
		\item The image of $H$ in the symmetric group $\Sym (A_i)$ equals $\AGL_1(q)$ embedded via its action on the affine line.
		\item There exists an element $h \in H$ that acts on each $A_i$ as a $(q-1)$-cycle (fixing one point).
		\item $H_i$ preserves the decomposition $\Omega=A_1 \sqcup ... \sqcup A_s \sqcup B$, fixes $B \cup A_i$ pointwise, and acts transitively on each $A_j$ for $j \neq i$.  
		\end{itemize}
	\end{enumerate}
\end{lemma}
\begin{proof}
	Let $U \subset (\PP^2)^*$ be a dense open subset over which the covering $\pi_X$ is \'{e}tale. The complement of $U$ contains finitely many lines. A line in $(\PP^2)^*$ defines a point in $(\PP^2)^{**}=\PP^2$. Let $S \subset \PP^2$ be the set of points corresponding to lines in the complement of $U$. Then for every point $P \not\in S$ the restriction of $\pi_X$ to the family of lines passing through $P$ is generically \'{e}tale. A general tangent to $X$ is tangent at exactly $s$ points with multiplicity $\max(2,q)$ at each one. Since $X$ is not strange, a general tangent line does not intersect the set of singular points $\mathrm{Sing}_X$ of $X$ and does not intersect $S$.  Choose a tangent line $\ell$ to $X$ that does not intersect $S \cup \mathrm{Sing}_X$ and such that $\ell \cap X$ has $s$ points with multiplicity $\max(2,q)$ and $d-qs$ points with multiplicity $1$. In particular for every point $P$ on $\ell$ the restriction of $\pi_X$ to the family of lines passing through $P$ is generically \'{e}tale.
	
	 Let $P_1, ..., P_s$ be the points of tangency and $Q_1, ..., Q_{d-qs}$ be the rest of the intersection. Choose a line in $\PP^2$ that is disjoint from both $\bigcup_i P_i$ and $\bigcup_j Q_j$. Consider the affine plane $\Aff^2$ obtained by removing the chosen line. Choose coordinates $x,y$ in $\Aff^2$ so that $\ell$ is the line $y=0$, and the coordinates of $P_1, ..., P_s$ are $(\alpha_1,0), (\alpha_2, 0),..., (\alpha_s, 0)$. Choose an algebraic closure of the fraction field of the local ring at $[\ell] \in (\PP^2)^*$; identify $\Omega$ with the fiber at the geometric point corresponding to this choice. Let $A_i \subset \Omega$ be the set that reduces to $P_i$, and let $B=\Omega \setminus \bigcup_i A_i$ be the rest of $\Omega$. For an element $\alpha \in K$ consider the family of lines $\ell_t=\ell_t(\alpha)$ given by the equation $y=t(x-\alpha)$. The family $\ell_t(\alpha)$ is the family of lines passing through $(\alpha,0)$, so the covering $\pi_X$ is generically \'{e}tale above the line corresponding to $\ell_t(\alpha)$ in $(\PP^2)^*$. We want to analyze inertia groups at $t=0$ of the restriction of $\pi_X$ to these families.
	
	\begin{enumerate}
		\item\label{lemmachar0} Assume $q=1$. In this case $\Char K \neq 2$, see \cite{Katz1973}*{Proposition~3.3}. We need to describe the extension of $\Klaur{t}$ obtained by adjoining the coordinates of the intersection of $X$ with $y=t(x-\alpha)$. Choose a uniformizer $x_i=x-\alpha_i$. For each point $P_i$ consider the local uniformization $y=p_i(x_i)$ of $X$ with respect to $x_i$, where $p_i(x_i)$ is a power series with coefficients in $K$. Since $q=1$ the power series $p_i(x_i)$ vanishes at $0$ with multiplicity $\max(q,2)=2$. When $\alpha \neq \alpha_i$ the Newton polygon of   $Q_i=p_i(x_i)-tx_i-t\alpha+t\alpha_i$ has a unique positive slope segment, it connects $(0,1)$ and $(2,0)$. Hence the positively valued roots of $Q_i$ have valuation $1/2$. Therefore when $\alpha \neq \alpha_i$ the extension obtained by adding the positively valued roots of $Q_i$ to $\Klaur{t}$ is the (unique) ramified degree $2$ extension. Similarly, if $\alpha = \alpha_i$ the equation $Q_i=0$ has two positively valued roots, both lying in $\Klaur{t}$. Taking $\alpha$ to be distinct from each $\alpha_i$ gives an element $g \in G_X$ that fixes $B$ and acts on each $A_i$ as a transposition. Taking $\alpha=\alpha_1$ gives an element $g_1 \in G_X$ that acts by identity on $B \cup A_1$ and as a transposition on $A_j$ for all $j \neq 1$. The product $gg_1$ is a transposition.
		
		\item Now assume $q>1$ is a power of the characteristic. This statement is similar to part \ref{lemmachar0}, except the degrees of the extensions are no longer $2$ and some wild ramification appears.  Take $x_i=x-\alpha_i$ as a uniformizer for $X$ at $P_i$. Let $(x_i,p_i(x_i))$ be the uniformization and define $c_i, d_i$ by $p_i(x_i)=d_i x^q + c_i x^{q+1}+O(x^{q+1})$. Choose $\alpha \in K$ so that both $d_i+(\alpha-\alpha_i)c_i \neq 0$ and $\alpha \neq \alpha_i$ hold for every $i$. 
		Consider the intersection of $X$ with the line $y=t(x-\alpha)$ over the field $\Klaur{t}$. The $x$-coordinates of the intersection points that reduce to $P_i$ are the positively valued roots of $Q_i(x)=p_i(x_i)-tx_i+t(\alpha-\alpha_i)$. Since $\alpha \not\in S$, the intersection of $X$ with $y=t(x-\alpha)$ over $K(t)$ defines a separable extension of $K(t)$. Therefore the positively valued roots of $Q_i$ lie in a separable extension of $\Klaur{t}$. Lemma~\ref{AGL1} applies to the power series $Q_i$ (condition $v(a_1a_q-a_0a_{q+1})=1$ becomes $d_i+(\alpha-\alpha_i)c_i \neq 0$). Let $L_i/\Klaur{t}$ be the extension obtained by adjoining all positive valuation roots of $Q_i$ to $\Klaur{t}$. Then the Galois group of $L_i/\Klaur{t}$ is $\AGL_1(q)$. The ramification index of $L_i/\Klaur{t}$ is divisible by $q(q-1)$. In particular every field $L_i$ is an overfield of $\Klaur{t^{1/(q-1)}}$. Let $H$ be the Galois group of the compositum of $L_i, i=1,...,s$. Take an element $h \in H$ that surjects onto a generator of $\Gal(\Klaur{t^{1/(q-1)}}/\Klaur{t})$. Every such element $h$ will surject onto an element of order $q-1$ on each $\Gal(L_i/\Klaur{t})$. We now examine the possibilities for the action of $h$ on the positive valued roots of $Q_i$. Assume that for some fixed $i$ the element $h$ has an orbit of length less than $q-1$ on the roots of $Q_i$. The extension $L_i/\Klaur{t}$ is generated by adjoining any two roots of $Q_i$. Assume that $h$ has an orbit of length $k$ with $1<k<q-1$. Then $h^k$ fixes at least two roots and therefore is the identity element in $\Gal(L_i/\Klaur{t})$. Contradiction. Therefore $h$ acts on each $A_i$ by fixing one point and moving the rest cyclically.

	\end{enumerate}
\end{proof}

\begin{remark}\label{introzero}
	When $\Char K = 0$ Lemmas~\ref{inertia} and~\ref{two-trans} imply that $G$ is $2$-transitive and contains a transposition, so $G=S_d$.
\end{remark}

\section{Sectional monodromy groups of nonstrange curves}\label{spacesection}

The main result of this section is Theorem~\ref{SpaceMundanity} that describes sectional monodromy groups of nonstrange space curves. 
The sectional monodromy group of a nonstrange space curve is triply transitive by Proposition~\ref{r-transitive}. Triply transitive permutation groups can be classified; for a statement and definitions of relevant groups see \cite{Cameron1999}*{Sections~7.3,~7.4}. 

\begin{proposition}\label{triply}
	If $G$ is a triply transitive permutation group of degree $d$ then either $G \supset A_d$ or one of the following holds.
	\begin{enumerate}
		\item $G=\AGL_n(2)$, $d=2^n$;
		\item $G=G_1$, $d=16$;
		\item $\PSL_2(r) \subset G \subset \PGaL_2(r)$, $d=r+1$, $r$ is a power of a prime;
		\item $G=M_{11}$, $d=11$;
		\item $G=M_{11}$, $d=12$;
		\item $G=M_{12}$, $d=12	$
		\item $G=M_{22}$, $d=22$;
		\item $G=\Aut(M_{22})$, $d=22$;
		\item $G=M_{23}$, $d=23$;
		\item $G=M_{24}$, $d=24$. 	
	\end{enumerate}
\end{proposition}

\begin{remark}
	The group $G_1$ is a subgroup of $\AGL_4(2)$. It is a semidirect product of the translation group $\F_2^4$ and a specially embedded copy of $A_7 \subset \GL_4(2)\simeq A_8$.
\end{remark}

 We will also need two geometric statements before proving Theorem~\ref{SpaceMundanity}.
 
 \begin{proposition}[Kaji]\label{Gauss-proj}
 	Let $X \subset \PP^r$, $r \geqslant 3$ be a projective integral nonstrange curve, let $\upgamma_X\colon X \to X^*$ be its Gauss map. Let $Y \subset \PP^{r-1}$ be the projection of $X$ from a general point $P \in \PP^r$. Identify $K(X)$ and $K(Y)$ via the projection from $P$. Let $\upgamma_Y\colon Y  \to Y^*$ denote the Gauss map of $Y$. Then the subfield $K(X^*)$ of $K(X)$ defined by $\upgamma_X$ and the subfield $K(Y^*)$ of $K(X)=K(Y)$ defined by  $\upgamma_Y\colon Y$ coincide.
 \end{proposition}
\begin{proof}
	See \cite{Kaji1989}*{Theorem~2.1} and the remark following the theorem.
\end{proof}
\begin{proposition}[\cite{Ballico-Cossidente1999}*{Theorem~3.1}]\label{skew-tangents}
	Let $X \subset \PP^3$ be a smooth curve of degree $d$. Assume that for a general point $P$ of $X$ there is no tangent line to $X$ at a point $Q \neq P$ that intersects the tangent line at $P$. Then either $d=3$ and $X$ is a twisted cubic or $d=q+1$ for some power $q$ of the characteristic. In the latter case $X$ is rational and isomorphic to the projective closure of the parametrized curve $(t,t^{q}, t^{q+1})$.
\end{proposition}
The following theorem is the main result of this section.
\begin{theorem}\label{SpaceMundanity}
	Let $X \subset \PP^m$, $m \geqslant 3$ be a nondegenerate integral nonstrange curve of degree $d$. Then one of the following holds.
	\begin{enumerate}
		\item The sectional monodromy group $G$ contains the alternating group $A_d$.
		\item $m=3$ and $X$ is projectively equivalent to the rational curve given as the projective closure of the parametrized curve $(t, t^q, t^{q+1}) \subset \Aff^3$ for some power $q$ of the characteristic. In this case $G = \PGL_2(q)$.
		\item\label{SpaceCase3} $m=3$, $\Char K=2$, the tangent variety of $X$ is a smooth quadric and $G$ is contained in $\AGL_n(2)$.		
	\end{enumerate}

\end{theorem}
\begin{proof}
	  Let $s,q$ be the separable and inseparable degrees of the Gauss map $\upgamma_X$ of $X$. Recall that for a general tangent line $\ell$ of $X$ the degree of the scheme-theoretic intersection of $X$ and $\ell$ satisfies $\deg X \cap \ell = \deg \upgamma = sq$. Choose one such general tangent line $\ell$ and choose a point $P$ of $X$ that does not lie on $\ell$. Let $H$ be a hyperplane containing $\ell$ and $P$. The degree of the intersection of $H$ with $X$ satisfies $d=\deg X = \deg H \cap X \geqslant \deg \ell \cap X + 1 = sq+1$. Therefore $sq<d$.
	  
	    By Proposition~\ref{Gauss-proj} $s$ and $q$ are also the separable and inseparable degrees of a general projection of $X$. Let $X'$ be a general projection of $X$ to $\PP^2$. Proposition~\ref{G-under-proj} implies that $G_X=G_{X'}$ as permutation groups. Since $X'$ is a plane curve, Lemma~\ref{inertia} applies. We conclude that if $q=1$ the group $G$ contains a transposition, and therefore $G=S_n$. Assume $q>1$ and fix a decomposition $\Omega=A_1 \sqcup ... \sqcup A_s \sqcup B$, subgroups $H, H_i \subset G$, and an element $h \in H$ as in Lemma~\ref{inertia}. We will use the properties of $(A_i, B, H_i, H, h)$ stated in Lemma~\ref{inertia} throughout this proof. Since $sq<d$ the set $B$ is nonempty.
	  
	 The group $G$ is triply transitive by Proposition~\ref{r-transitive}. Therefore it is on the list of Proposition~\ref{triply}. We analyze each possibility.
	\begin{enumerate}[label=\textsf{Case (\alph*):}]
		\item $G\subset \AGL_n(2)$ and $\Omega=\F_2^n$. This case also covers the $G_1$ group of Proposition~\ref{triply}, since $G_1$ is a permutation subgroup of $\AGL_4(2)$. Assume first that $q \neq 2$. The element $h \in G$ acts as a product of $s$ $(q-1)$-cycles. The element $h$ has a fixed point, which we may assume is the origin in $\F_2^n$, so $h \in \GL_n(\F_2)$. Consider $\Omega = \F_2^n$ as an $\F_2[t]$ module, where $t$ acts as $h$. Since $h$ has fixed points and is not equal to the identity, one of the following cases holds.
		 
	\begin{enumerate}[label=(\roman*)]
		\item  $\Omega$ has a submodule $W$ isomorphic to $\F_2[t]/(t-1)^n$ for some $n>1$. Let $M$ be the restriction of $h$ to $W$. Then $M^{q-1}-1=0$. So $t^{q-1}-1$ is divisible by $(t-1)^n$ in $\F_2[t]$, therefore $n \leqslant v_2(q-1)$. This implies that the number of points in $W$ is $2^n \leqslant q-1$. But $q-1$ is the minimal size of a nontrivial orbit of $h$, contradiction.
		
		\item $\Omega=V_1 \oplus V_2$, where $V_1 \neq 0$ is indecomposable and $V_2$ contains all nonzero fixed points of $h$. Then for every fixed point $v$ of $h$, $v+V_1$ is invariant under the action of $h$. Therefore, $h$ has at least as many orbits of length greater than $1$ as it has fixed points. On the other hand, since $h$ acts on $A_i$ via fixing one point and permuting the rest cyclically and also fixes $B \neq \emptyset$ pointwise, $A$ has more fixed points than it has orbits of length greater than one. Contradiction. 
		\end{enumerate}
		Assume $q=2$. The set $B$ is the set of common fixed points of the elements of $H$. Therefore $B$ forms an affine space, and $\#B$ is a power of $2$. Similarly $B \cup A_1$ is the set of common fixed points of elements of $H_1$. Therefore $\#B + \#A_1=\#B+2$ is a power of $2$. Hence $\#B=2$, and $2^n=\# \Omega = 2s + \# B$, so $s=2^{n-1}-1$. We can assume that the tangent variety of $X$ is not a quadric, since otherwise we are in case~\ref{SpaceCase3}. If $n=2$, then a nonidentity element of $H$ acts on $\Omega$ as a transposition, hence $G=S_d$. Suppose $n>2$. Then $X$ has degree $2^n>4$ and therefore $X$ lies on at most one quadric surface. For a general tangent line $\ell$ of $X$ there does not exist a quadric $Q$ containing $X$ and $\ell$. Consider a triple $\ell_1, \ell_2, \ell_3$ of general tangents to $X$. There exists a smooth quadric $Q$ containing $\ell_1, \ell_2, \ell_3$ and not containing $X$. Then 
		\begin{multline*}2^{n+1}=\deg Q \cap X \geqslant \deg \ell_1 \cap X + \deg \ell_2 \cap X + \deg \ell_3 \cap X\\
		 = 3 \deg \upgamma_X = 6s=6 \cdot (2^{n-1}-1).\end{multline*}  This inequality is never satisfied for $n>2$, contradiction.

		\item $\PSL_2(r) \subset G \subset \PGaL_2(r)$ and $\Omega=\PP^1(\F_r)$. In this case $\deg X=\#\Omega=r+1$. The group $\PGaL_2(r)$ is generated by $\PGL_2(r)$ and the Frobenius. In most of the cases analyzed below, the subgroups $H, H_i$ contain elements that fix a triple of points. The group $\PGaL_2(r)$ is triply-transitive; the pointwise stabilizer of $\{0,1,\infty\}$ in $\PGaL_2(r)$ is the cyclic group generated by the Frobenius, and so the pointwise stabilizer of any triple of points is cyclic.  Suppose $q < r$. We consider several subcases.
		\begin{enumerate}[label=(\roman*)]
			\item $s=1$. In this case the set $B$ of Lemma~\ref{inertia} contains $r+1-q \geqslant 2$ elements.
			If $q=2$, then a nontrivial element of $H$ is a transposition, so $G=S_d$.
			Suppose $\#B=r+1-q=2$ and $q>2$. Then the element $h \in H$ acts as a $(q-1)$-cycle and fixes $3$ points. The stabilizer of a triple of points in $\PGaL_2(r)$ is a cyclic group generated by the Frobenius, and thus cannot have exactly $3$ fixed points when $q>2$. Therefore we can assume that $\#B=r+1-q \geqslant 3$ and $q>2$. The group $H$ stabilizes $B$ pointwise; therefore $H$ is abelian. On the other hand, Lemma~\ref{inertia} implies that $H$ maps surjectively onto $\AGL_1(q)$, contradiction.
			
			\item $s \geqslant 2, q>2$. The element $h \in \PGaL_2(r)$ fixes $s+\# B \geqslant 3$ points, and therefore  $h$ is a power of the Frobenius. Since $h$ only has orbits of sizes $1$ and $q-1$, the number $q-1$ must be prime. Thus $\F_r=\F_{\ell^{q-1}}$ for some prime power $\ell$. The action of $h$ on $\F_{\ell^{q-1}}$ is such that the number of fixed points is $s+\#B-1$ which is at least as large as the number of nontrivial orbits $s$. So $(\ell^{q-1}-\ell)/(q-1)=\#\{\mathrm{nontrivial\ orbits\ of\ }h\} \leqslant \#\{\mathrm{fixed\ points\ of\ } h\} =\ell+1$. The only pairs of prime powers $(\ell, q)$ with $q>2$ that satisfy this inequality are $(2,4)$, $(2,3)$, and $(3,3)$. These correspond to $(r,q)$ being $(8,4)$, $(4,3)$, and $(9,3)$. The case $(r,q)=(4,3)$ is impossible since $s>1$. Suppose $(r,q)=(9,3)$. The element $h$  must have exactly $3+1$ fixed points, and therefore $s=3$. Any element $g \in H_1$, $g \neq \mathrm{id}$ fixes $B \cup A_1$. Since $\# B + \# A_1 \geqslant 3$, the element $g$  must equal to the Frobenius in a degree $q-1=2$ extension, and therefore be a product of transpositions. However elements of $H_1$ preserve the decomposition $A_2 \cup A_3$ and act transitively on $A_2$, contradiction. Similarly, in the case $(r,q)=(8,4)$ the element $h$ must have $3$ fixed points, and hence $s=2$. In this case any nontrivial element of $H_1$ is a power of the Frobenius, and hence $H_1$ cannot act transitively on $A_2$, contradiction.
			
			\item $q=2$, $s \neq 1$. If $s=2$, then a nontrivial element of $H_1$ is a transposition, proving $G=S_d$. In particular we can assume $s \geqslant 3$, $r \geqslant 7$.
		 Choose three elements of $B \cup A_i$ and identify them with $0, 1, \infty$. A nontrivial element of $g \in H_i$ fixes $0,1, \infty$ and has order $2$, therefore $g$ is equal to the Frobenius of the quadratic extension $\F_r/\F_\ell$, $r=\ell^2$. Hence $H_i=\mathbb{Z}/2\mathbb{Z}$. Since $H_i$ acts transitively on $A_j$ for $j \neq i$, the nontrivial element $g \in H_i$ acts as a transposition on $A_j, j \neq i$. In particular $\#B+2=\# \PP^1(\F_\ell)=\ell+1$. Suppose $B$ has at least three points. Label them $0, 1, \infty$. Every element of $H$ fixes $0,1, \infty$ and has order at most $2$, so $H=\mathbb{Z}/2\mathbb{Z}$. Since $H$ acts transitively on each $A_i$, the set of fixed points of $H$ is $B$. Therefore $\#B=\ell+1=\#B+2$. Contradiction. We proved that $\#B \leqslant 2$ and $H_i$ is a group of order $2$ fixing $B \cup A_i$ and acting as a transposition on $A_j$ for $j \neq i$. Since $\ell+1=\#B+2 \leqslant 4$, either $\ell=2$ or $\ell=3$. Since $\ell^2=r \geqslant 7$, we have $\#B=2$, $\ell=3$, $s=4$. Let $h_i$ denote the nontrivial element of $H_i$. Then $h_1h_2$ fixes $A_3 \cup A_4 \cup B$ and acts on $A_1$ and $A_2$ as a transposition. No element of $\PGaL_2(9)$ fixes exactly $6$ points. Contradiction. 

		\end{enumerate}
		We have shown that $q$ cannot be less than $r$. Suppose $q=r$. The group $\PGaL_2(q)$ is not quadruply transitive unless it contains $A_{q+1}$ (i.e., unless $q=2,3,4$). Therefore Proposition~\ref{r-transitive} implies that the curve $X$ lies in $\PP^3$. Take a smooth point $P$ of $X$ and consider the family of hyperplanes passing through the tangent line to $X$ at $P$. Since the multiplicity of the intersection of such a hyperplane with $X$ at $P$ is at least $q$ and $\deg X = q+1$, the curve $X$ is smooth. If a general tangent to $X$ is concurrent with another tangent line, then the plane containing both of them intersects $X$ with multiplicity greater than $q+2>q+1$ which is impossible. Therefore we can apply Proposition~\ref{skew-tangents}. We conclude that $X$ is projectively equivalent to the projective closure of the rational curve $(t, t^q, t^{q+1})$. The monodromy group in this case is $\PGL_2(q)$, see \cite{Rathmann1987}*{Example~2.15}.
	\item $G$ is one of the Mathieu groups and the action on $\Omega$ is in the list of Theorem~\ref{Jones}. 
	 From Lemma~\ref{inertia} it follows that there is an element in $G$ that acts via a product of $s$ $(q-1)$-cycles. Take an element $g \in G$ that maps onto an element of order $p$ in $\AGL_1(p)$. Replacing $g$ by $g^k$ is necessary we can assume that $g$ acts as a nontrivial product of cycles of length $p$. Cycle types of the Mathieu groups (in different permutation representations) are listed in Appendix~\ref{MathieuCycles} (borrowing from \cite{ATLASv3}). Together the existence of cycle type $(q-1)^s$, existence of a product of $p$-cycles, and the condition $sq < d$ leaves the following cases.
	 \begin{enumerate}[label=(\roman*)]
	 	\item $G=M_{11}$, $d=11$, $q=2$, $s=4$. Any nontrivial element of $H_1$ has cycle type $2^t$, $1 \leqslant t \leqslant 3$. However there are no element with such cycle type in this representation.
	 	\item $G=M_{11}$, $d=12$, $s=5$, $q=2$. All elements having cycle type $2^t$, $t>0$ have cycle type $2^4$. These cannot form the subgroup $H$, since a product of two distinct such elements in $H$ is an element of type $2^2$.
	 	\item $G=M_{11}$, $d=11$, $q=5$, $s=2$. Every element of $H_1$ fixes at least $6$ points; $M_{11}$ has no such elements.
	 	\item $G=M_{11}$,  $d=12$, $q=2$, $s=5$. Any element of type $2^t$ with $t>0$ has $t=4$. Therefore every element of $H$ is a product of $4$ transpositions. A product of two distinct elements acting as products of four $2$-cycles is not a product of four $2$-cycles.
	
	 	\item $G=M_{12}$, $q=2$, $s \leqslant 5$. Any element of $H_1$ acts as a product of at most $s-1$ transpositions. Every element of $M_{12}$ with cycle type $2^t$ has $t=0, 4, 6$. Therefore $s=5$. Every element of $H$ acts as a product of at most $5$ transposition, so every nonidentity element of $H$ acts as a product of $4$ transpositions. A product of two $2^4$ elements of $H$ cannot be a $2^4$ element. Contradiction.
	 	\item $G=M_{12}$, $q=5$, $s=2$. The group $H_1$ contains a nonidentity element that fixes at least $7$ points; $M_{12}$ has no such elements.
		\item $G=M_{22}$, $q=2$, $s \leqslant 10$. Any element of $H_1$ acts as a product of transpositions. An element of $M_{22}$	that has cycle type $2^t$ has $t=0, 8$. Therefore $s \geqslant 9$. Every element of $H$ is a $2^8$ element. A product of $2$ distinct $2^8$ elements of $H$ cannot be a $2^8$ element, contradiction. 
	 	\item $G=\Aut(M_{22})$, $q=2$, $s \leqslant 10$. Any element of $H_1$ acts as a product of at most $s-1 \leqslant 9$ transpositions. An element of $\Aut(M_{22})$	that has cycle type $2^t$ has $t=0, 7,8$. Therefore $s \geqslant 8$.  Suppose $s=8$. Then $H_1$ acts transitively on $A_i$, for $i>1$. Since the only possible cycle type of a nonidentity element of $H_1$ is $2^7$, $H_1=\mathbb{Z}/2\mathbb{Z}$ acting as a transposition on every $A_i$, $i \neq 1$. Consider now the possible cycle types of elements of $H$. If $H$ contains a $2^8$ element $g$, then the product of $g$ with the nonidentity element of $H_1$ is a transposition, contradiction. On the other hand, a product of two different $2^7$ elements of $H$ cannot be a $2^7$ element. Therefore $s \neq 8$. Suppose $s>8$. Then the group $H$ contains two disctinct elements $g_1, g_2$ with cycle types $2^{t_1}, 2^{t_2}$ for $t_1, t_2 \in \{7,8\}$. Then $g_1g_2$ is not the identity, and therefore fixes at most $2*(s-7)$ elements of $\cup_i A_i$. On the other hand, $g_1$ and $g_2$ both act as a transposition on $A_i$ for at least $t_1+t_2-s \geqslant 14-s$ different $i$. Therefore $t_1+t_2-s \leqslant s-7$, which implies $s>10$. Contradiction.

	 	\item $G=\Aut(M_{22})$, $q=3$, $s=7$. Any element of the group generated by $H_1, ..., H_7$  has cycle type $2^k3^\ell$, with $k, \ell \leqslant 7$. The only cycle types of this form in $\Aut(M_{22})$ are $3^6$ and $2^7$. Since $H_i$ fixes $A_i$, an element of $H_i$ cannot have cycle type $2^7$, so every nonidentity element of $H_i$ has cycle type $3^6$. Consider nonidentity elements $g_1 \in H_1, g_2 \in H_2$. The product $g_1g_2$ acts as a $3$-cycle on $A_1$, therefore $g_1g_2$ has cycle type $3^6$. This means that $g_1$ and $g_2$ restrict to the same $3$-cycle on $A_i$ for $4$ different values of $i$. Then $g_1g_2^{-1}$ fixes at least $3 \cdot 4 = 12$ points, however $\Aut(M_{22})$ has no nonidentity elements fixing $12$ points.
	 	
	 	\item $G=M_{23}$,  $d=23$, $q=2$. Every element with cycle type $2^t$, $t>0$ has $t=8$. Since every element of $H_1$ is a product of at most $s-1$ transpositions, $s \geqslant 9$. Also $sq<d$ implies $s\leqslant 11$. Since $H$ acts transitively on each $A_i$, there exist two distinct nontrivial elements $h_1, h_2 \in H$ each having the cycle type $2^8$. But $h_1h_2$ cannot have cycle type $2^8$, contradiction.
	 	\item $G=M_{24}$, $d=24$, $q=2$, $s \geqslant 8$. Since $sq < d$, we have $s \leqslant 11$. In this representation every $2^t$ element of $H$ has $t=0,8$. A product of two $2^8$ elements of $H$ cannot be a $2^8$ element. Contradiction.
	  	\end{enumerate}

		\end{enumerate}
\end{proof}

\begin{remark}\label{RathmannAn}
Suppose $\Char K \neq 2$. Then the group $G$ is contained in the alternating group if and only if the separability degree of the Gauss map is even \cite{Rathmann1987}*{Theorem~2.10}. When $\Char K = 2$ we do not know under what conditions $G$ is equal to $S_n$.
\end{remark}
\section{Galois groups of trinomials}\label{trinomials}

In this section we study the Galois group $G$  of the polynomial $x^n + ax^m + b$ over the field $K(a,b)$. The following theorem of Cohen covers most of the previously known results about $G$. 

\begin{theorem}[\cite{Cohen1980}*{Corollary~3}]\label{Cohen}
	Assume that $n$ and $m$ are relatively prime, $p \nmid m(n-m)$ and if $m=1$  or $m=n-1$ assume additionally that $p \nmid n$. Then $A_n \subset G$. Moreover, if $p$ is odd, then $G = S_n$.
\end{theorem}

Some cases not covered by Theorem~\ref{Cohen}, such as a trinomial with $G=M_{11}$, are described in \cite{Uchida1970}.
Our goal is to compute all the possibilities for $G$. 

We begin by fixing notation. Throughout this section, $(n,m)$ is a pair of relatively prime positive integers, $X_{n,m}$ denotes the projective rational curve plane $x^n=y^mz^{n-m}$, and $G=G_{n,m}$ is the sectional monodromy group of $X_{n,m}$. An affine parametrization of $X_{n,m}$ is $(t^m,t^n)$. Therefore $G$ is also the Galois group of the trinomial $x^n+ax^m+b$ over $K(a,b)$. Without loss of generality we assume that $m < n/2$ since $X_{n,m} \simeq X_{n,n-m}$. 

We start by describing properties of tangent lines and Gauss maps of the curves $X_{n,m}$.

\begin{proposition}\label{Gauss-trinomial}
\hfill
\begin{enumerate}
	\item If $n,m$ and $(n-m)$ are prime to $p$, then $X_{n,m}$ is nonstrange. In this case the Gauss map has separability degree $1$, and inseparability degree $1$ when $p>2$ and $2$ when $p=2$.
	
	\item Assume that $k \in \{n,m,n-m\}$ is divisible by $p$. Write $k=qd$ where $q$ is the largest power of $p$ that divides $k$. Then $X_{n,m}$ is strange and the Gauss map has inseparability degree $q$ and separability degree $d$.
	
\end{enumerate}	
\end{proposition}
\begin{proof}
	 If $p|m$, then $p \nmid n$, so we can replace $X_{n,m}$ by the projectively isomorphic curve $X_{n,n-m}$, and $p \nmid n-m$. Thus we may assume $p \nmid m$. An affine equation of $X_{n,m}$ is $x^n=y^m$. The tangent line to $X_{n,m}$ at the point $(\alpha^m, \alpha^n)$  is given by the equation 
	\[ n\alpha^{m(n-1)} X - m\alpha^{n(m-1)}Y = (n-m)\alpha^{nm}.\]
	 Restricting the equation $x^n=y^m$ to the tangent line $Y= \frac{n}{m}\alpha^{n-m}X - (\frac{n}{m}-1)\alpha^n$ gives the polynomial
	 \[p(X)=X^n - \left(\frac{n}{m}\alpha^{n-m}X - \left(\frac{n}{m}-1\right)\alpha^n\right)^m\] Recall that the inseparability degree $r$ of the Gauss map is greater than $1$ if and only if $\alpha^m$ is a root of $p(X)$ with multiplicity greater than $2$; in this case $r$ is equal to the multiplicity of the root $\alpha^m$. Set $X=\alpha^m+Z$. Then 
	 \[p(X)=(Z+\alpha^m)^n - \left( \frac{n}{m}\alpha^{n-m}Z + \alpha^n\right)^m.\] The coefficient of $Z^2$ is \[\left(\frac{n(n-1)}{2} - \frac{n(m-1)}{2} \right)  \alpha^{m(n-2)}.\] This coefficient is identically zero if and only if either $p | n$ or $p | (n-m)$.
	 \begin{enumerate}
	 	\item $p \nmid n, p \nmid n-m$. In this case the inseparability degree of the Gauss map is $1$. We can recover the point $(\alpha^n, \alpha^m)$ from its tangent line	\[ n\alpha^{m(n-1)} X - m\alpha^{n(m-1)}Y = (n-m)\alpha^{nm}\] using  ratios of the coefficients of the equation. Therefore the separability degree of the Gauss map is also $1$. Finally the dual curve $(na^{m(n-1)}:m\alpha^{n(m-1)}:(n-m)\alpha^{nm})$ is not a line; therefore the tangent lines do not pass through one common points and the curve is nonstrange.
		 	
	 	\item $p|n$. Let $q$ be the largest power of $p$ dividing $n$, write $n=qd$. In this case the equation of the tangent line becomes $Y=\alpha^n$, and therefore the curve is strange. The Gauss map in affine coordinates is $\alpha \to \alpha^n$; therefore it has separable degree $d$ and inseparable degree $q$.
		 	
	 	\item\label{n-m} $p| (n-m)$. Let $q$ be the largest power of $p$ dividing $n-m$, and let $n=qd$. The equation of the tangent line becomes $Y=\alpha^{n-m}X$. Therefore the curve is strange. The Gauss map in affine coordinates is just $\alpha \to \alpha^{n-m}$ therefore the separable degree of the Gauss map is $d$ and the inseparable degree is $q$.  
	 \end{enumerate}

\end{proof}

We now prove a stronger version of Lemma~\ref{inertia} for curves $X_{n,m}$.

\begin{lemma}\label{trinomial-inertia}
	Let $X=X_{n,m}$ be one of the trinomial curves. Choose a geometric point $\delta$ of $(\PP^2)^*$ above which the covering $\pi_X$ is \'{e}tale. Denote by $\Omega$ the fiber of $\pi_X$ above $\delta$, so the group $G=G_X$ acts on $\Omega$ by permutations. Then the following hold.
	\begin{enumerate}
		\item Assume $p|n$. Write $n=qd$ with $q$ a power of $p$ and $d$ prime to $p$. Then there exists a decomposition $\Omega = A_1 \sqcup ... \sqcup A_d$, a subgroup $H \subset G$, a collection of subgrops $H_i \subset G$, $i=1,...,d$, and an element $h \in H$ with the following properties.
		\begin{itemize}
			\item For every $i$ the set $A_i$ is of cardinality $q$.
			\item The image of $H$ in $\mathrm{Sym}(A_i)$ equals $\AGL_1(q)$ embedded via its action on $\Aff^1(\F_q)$.
			\item The element $h \in H$ acts on each $A_i$ with one fixed point and one orbit of size $q-1$.
			\item The group $H_i$ acts on $A_i$ with one fixed point and one orbit of size $q-1$. The image of $H_i$ in $\mathrm{Sym}(A_j)$ for $i \neq j$ contains $\AGL_1(q)$ embedded via its action on $\Aff^1(\F_q)$.
		\end{itemize}
	\item Assume $p | n-m$. Write $n-m=qd$ with $q$ a power of $p$ and $d$ prime to $p$. Then there exists a decomposition $\Omega = A_1 \sqcup ... \sqcup A_d \sqcup B$, a subgroup $H \subset G$, a collection of subgrops $H_i \subset G$, $i=1,...,d$, and an element $h \in H$ with the following properties.
	\begin{itemize}
		\item For every $i$ the set $A_i$ is of cardinality $q$.
		\item The image of $H$ in $\mathrm{Sym}(A_i)$ equals $\AGL_1(q)$ embedded via its action on $\Aff^1(\F_q)$.
		\item The element $h \in H$ acts on each $A_i$ with one fixed point and one orbit of size $q-1$.
		\item The image of $H$ in $\mathrm{Sym}(B)$ is $\mathbb{Z}/m\mathbb{Z}$ in its natural embedding; $h$ maps to a generator of $\mathbb{Z}/m\mathbb{Z}$.  
		\item The group $H_i$ fixes $A_i \cup B$ and acts transitively on each $A_j$ for all $j \neq i$.
	
	\end{itemize}
	\end{enumerate} 
\end{lemma}
\begin{proof}

		Both parts of the statement are proved similarly to Lemma~\ref{inertia}. We will pick a tangent line $\ell$ to $X_{n,m}$. We will then restrict the covering to a family of lines that includes $\ell$ and compute the inertia group at $[\ell] \in (\PP^2)^*$. To calculate the inertia groups we will use Lemma~\ref{AGL1} and a computation of local uniformization.
	\begin{enumerate}
			
	\item Replacing $K$ with a larger algebraically closed field if necessary, we can choose an element $u \in K$ that is not a root of unity. Consider the affine model $x^n=y^m$ of $X$. The line $y=u^n$ is tangent to $X$. Choose a root of unity $\zeta_d \in K$. The points of tangency are $P_i=(\alpha_i,u^n)$, where $\alpha_i=\zeta_d^iu^m$. Let $x_i=x-\alpha_i$. A local uniformization of $X$ at $P_i$ is given by $(\alpha_i+x_i, p_i(x_i))$, where $p_i(x_i)=u^n(1+x/\alpha_i)^{n/m}$. The $x^{q+1}$  coefficient of $p_i$ is zero. Choose an element $\alpha \in K$ and identify $\Omega$ with the intersection of $X_{n,m}$ and the line $\ell$ given by the equation $y-u^n=t(x-\alpha)$ over the field $\overline{\Klaur{t}}$. Let $H(\alpha)$ be the Galois group of the field extension $L_\alpha$ obtained by adjoining to $\Klaur{t}$ the coordinates of the intersection. Then $H(\alpha)$ acts on $\Omega$ and is a subgroup of $G$. The field $L_\alpha$ is the field obtained by adding the roots of $x^n-tx^m+t\alpha-u^n$ to $\Klaur{t}$. In particular it is separable. Let $A_i \subset \Omega$ be the set of intersection points that reduce to $P_i$. The $x$-coordinates of the points in $\ell \cap X$ that reduce to $P_i$ are the positively valued roots of the equation $p_i(x_i)-u^n=tx_i + t\alpha_i-t\alpha$. Lemma~\ref{AGL1} applies with $a_0=t(\alpha-\alpha_i)$, $a_1=t$, $a_q=u^nd\alpha_i^{-q}/m$, $a_{q+1}=0$. Hence if $\alpha \neq \alpha_i$, the Galois group of the field $L_{i, \alpha}$ obtained by adding these roots to $\Klaur{t}$ equals $\AGL_1(q)$. When $\alpha=\alpha_i$ the equation has a rational root $x_i=0$ and the rest of the roots have valuation $1/(q-1)$. Therefore when $\alpha=\alpha_i$ the field $L_{i,\alpha}$ is the unique extension of $\Klaur{t}$ of degree $q-1$.  Let $H_i$ be the group $H(\alpha_i)$. Fix $\alpha \in K$ distinct from each $\alpha_i$ and let $H$ be the group $H(\alpha)$. We have shown that the actions of $H$ and $H_i$'s on $A_j$'s are as claimed. Now we need to show the existence of $h \in H$ that acts on each $A_i$ as a $(q-1)$-cycle. The extension $L_{i,\alpha}$ has ramification index $q(q-1)$ and therefore contains the field $\Klaur{t^{1/(q-1)}}$. Let $h \in H$ be an element that surjects onto a generator of $\Gal(\Klaur{t^{1/(q-1)}}/\Klaur{t})$. Then the order of $h$ acting on $A_i$ is $q-1$. Elements of $\AGL_1(q)$ of order $(q-1)$ act with two orbits of lengths $1$ and $q-1$. 
	
	\item Consider the affine model $x^n=y^{m}$ of $X$. Choose an element $u \in K$ that is not a root of unity. Let $\ell$ be the tangent line $y=u^{n-m}x$ to $X$. Choose a root of unity $\zeta_d$. The points of tangency are $P_i \colonequals (\zeta_d^i u^m, \zeta_d^i u^n)$. The intersection $\ell \cap X$ consists of points $P_i$ and the point $(0,0)$. For an element $\alpha \in K$ let $\ell_\alpha$ denote the line $y=(t+u^{n-m})x-\alpha t$ over the field $\Klaur{t}$.  Identify $\Omega$ with the intersection $\ell_\alpha \cap X$ over $\overline{\Klaur{t}}$. The absolute Galois group of $\Klaur{t}$ acts on $\Omega$ as a subgroup of $G$. Let $A_i \subset \Omega$ be the subset corresponding to points that reduce to $P_i$, and let  $B$ be the subset that corresponds to points reducing to $(0,0)$. Let $L_\alpha$ be the extension obtained by adding the coordinates of the intersection $\ell_\alpha \cap X$ to $\Klaur{t}$. Then $L_\alpha$ is the splitting field of the trinomial $x^n-(t+u^{n-m})x^m+\alpha t$. In particular $L_\alpha$ is separable. Consider the uniformizer $x_i=x-\zeta_d^iu^m$ at $P_i$. Let the local uniformization of $X$ at $P_i$ be given by $(\zeta_d^iu^m+ x_i, p_i(x_i))$. The power series $p_i(x_i)$ equals $\zeta_d^i u^n(1+x_i/(\zeta_d^iu^m))^{n/m}=\zeta_d^i u^n(1+x_i/(\zeta_d^iu^m))(1+x_i/(\zeta_d^iu^m))^{dq/m}$. The $x$ coordinates of the points in the intersection $\ell_\alpha \cap X$ that reduce to $P_i$ are the postiviely valued roots of the power series $Q_i=p_i(x_i)-(t+u^{n-m})(x_i+\zeta_d^iu^m) + \alpha t$. We have \[Q_i = a_0 + a_1 x_i + a_qx_i^q + a_{q+1}x_i^{q+1} + O(x^{2q}),\] where $a_0=(\alpha-\zeta_d^iu^m)$, $a_1=-t$, $a_q=\frac{d}{m}\zeta_d^{-i(q-1)}u^{-mq+n}$, $a_{q+1}=\frac{d}{m} u^{n-(q+1)m}\zeta_d^{-iq}$. For a general $\alpha$ the expression $a_0a_{q+1}-a_1a_q$ has valuation $1$. For each $i$ let $L_{\alpha, i}$ be the field generated by the positively valued roots of $Q_i$, and let $L_\alpha$ be the field obtained by adjoining to $\Klaur{t}$ the coordinates of the intersection of $X$ with $\ell_\alpha$. The field $L_\alpha$ is the compositum of $L_{\alpha, i}$ for $i=1,..., d$. Let $H(\alpha)$ be the Galois group of the extension $L_\alpha/\Klaur{t}$. Lemma~\ref{AGL1} implies that for a general $\beta \in K$ the image of $H(\beta)$ in $\Sym(A_i)$ is $\AGL_1(q)$ for all $i$. Fix one such $\beta\neq 0$ and let $H\colonequals H(\beta)$. 
	
	We now compute the action of $H(\alpha)$ on $B$ for $\alpha \neq 0$. The elements of $B$ are the roots of $Q(x)=x^n-(t+u^{n-m})x^m+\alpha t$ that reduce to $0$ modulo $t$. The Newton polygon of $Q(x)$ contains only one segment of positive slope, it connects $(0,1)$ and $(m,0)$. Since the slope is $1/m$ and $m$ is prime to $p$, the elements of $B$ are permuted cyclically by the group $H(\alpha)$ for every $\alpha\neq 0$. Let $h \in H$ be an element that surjects onto a generator of $\Gal(\Klaur{t^{1/LCM(q-1, m)}}/\Klaur{t})$. Since the field extension obtained by adjoining $B$ to $\Klaur{t}$ is $\Klaur{t^{1/m}}$, the element $h$ acts on $B$ through a cyclic permutation. The field $L_{\alpha,i}$ has ramification index $q(q-1)$ and therefore contains $\Klaur{t^{1/q-1}}$. Hence the action of $h$ on $A_i$ has order $q-1$, so  $h$ acts on each $A_i$ by fixing one point and permuting the rest cyclically.
	 
	Let $L_i$ denote the field $L_{\zeta_d^iu^m}$. Let $H_i$ be the Galois group of $L_i$ over $\Klaur{t^{1/\mathrm{LCM}(m,q-1)}}$. Then $H_i$ fixes $B$. When $\alpha=\zeta_d^iu^m$ the power series $Q_i$ has root $0$. The positive slope part of the Newton polygon of $Q_i/x_i$ is a single segment from $(0,1)$ to $(q-1, 0)$. Therefore $H_i$ fixes $A_i$. For $j\neq i$ the positive slope part of the Newton polygon of $Q_j$ is a segment from $(0,1)$ to $(q,0)$; therefore $H_i$ acts transitively on $A_j$ for $j \neq i$.    
	\end{enumerate}
\end{proof}
\begin{theorem}\label{final-trinomial}
Suppose $m$ and $n$ are relatively prime integers satisfying $m \leqslant n/2$.	Let $G$ be the Galois group of the polynomial $P(x)\colonequals x^n+ax^m+b$ over $K(a,b)$, and let $p=\Char K$.
	\begin{enumerate}
		\item If $m=1$ and $n=p^d$, then $G=\AGL_1(p^d)$.
		\item If $m=1$, $n=6$, and $p=2$, then $G=\PSL_2(5)$.
		\item If $m=1$, $n=12$, and $p=3$, then $G=M_{11}$.
		\item If $m=1$, $n=24$, and $p=2$, then $G=M_{24}$.
		\item If $m=2$, $n=11$, and $p=3$, then $G=M_{11}$.
		\item If $m=3$, $n=23$, and $p=2$, then $G=M_{23}$.
		\item Let $q\colonequals p^k$ for some $k$. If $n=(q^d-1)/(q-1)$, $m=(q^s-1)/(q-1)$ for some $s,d$, then $G=\PGL_d(q)$.
		\item If none of the above holds, then $A_n \subset G$.	
	\end{enumerate}
\end{theorem}
\begin{proof}
	We break the proof into cases corresponding to the cases of Proposition~\ref{Gauss-trinomial}. The case $n=2$ is trivial. In what follows we assume $n>2$ so that $m<n/2$ since $m,n$ are relatively prime.
\begin{enumerate}

\item Assume $p \nmid nm(n-m)$. By Proposition~\ref{Gauss-trinomial} the curve $X_{m,n}$ is nonstrange and the Gauss map is either birational or purely inseparable of degree two. In any case Lemma~\ref{inertia} implies that $G$ contains a transposition; hence $G=S_n$.

\item Assume $p | n$. Consider the specialization of  $x^n+ax^m+b$ to the polynomial $P(x)\colonequals x^n + t^{-1}x^m + 1$ over the field $\Klaur{t}$. The Newton polygon of $P$ consists of two segments: one from $(0,0)$ to $(m,-1)$ and one from $(m, -1)$ to $(n,0)$. Since $n-m$ and $m$ are relatively prime and are both prime to $p$, the group $G$ contains an $m$-cycle and an $(n-m)$-cycle. If $m \neq 1$, then $A_n \subset G$ by Theorem~\ref{Jones}. If $m=1$ and $A_n \not\subset G$, then Theorem~\ref{Jones} implies that one of the following holds
\begin{enumerate}
	\item $n=p^d$ \label{n-a}. In this case the Galois group is $\AGL_1(p^d)$, see \cite{Rathmann1987}*{Example~2.17}.
	
	\item $n=\ell+1$ for some prime $\ell$ and $G=\PSL_2(\ell)$ or $G=\PGL_2(\ell)$. Write $n=qd$ for $d$ prime to $p$ and a power $q$ of $p$. Proposition~\ref{Gauss-trinomial} states that in this case the Gauss map has inseparable degree $q$ and separable degree $d$. If $d=1$, then $n$ is a power of $p$, which is precisely the case~\ref{n-a} above. Assume $d > 1$. Lemma~\ref{trinomial-inertia} applies in this case with $B=\emptyset$. The group $H_1$ surjects onto $\AGL_1(q)$ and therefore has order divisible by $p$. Take an element $g_1 \in H_1$ of order $p$, so that it acts as a nontrivial product of cycles of length $p$ and fixes $A_1$. Since no element of $\PGL_2(\ell)$ can fix three points,  $\# A_1 =q=2$. The element $g_1$ fixes two points and acts as a products of transpositions. There is a unique such element, namely multiplication by $-1$. In particular $g_1$ acts as a transposition on $A_i$ for $i>1$. Take an element $g_2 \in H$ that acts on $A_1$ as a transposition. If $d \geqslant 4$, then either $g_2$ or $g_1g_2$ fixes at least four points, contradiction. If $d=2$, then $g_1$ acts a transposition and so $G=S_n$. If $d=3$, then $\ell=5$ (and $q=2$). Suppose  $\ell=5$, $q=2$. There is a factorization 
	\[x^6+ax+b = \left(x^3 + c x^2 + c^2\alpha x + \frac{a}{c^2}+\alpha c^3\right) \left(x^3+ cx^2 + (c^2\alpha+ c^2)x +\frac{a}{c^2}+(\alpha+1) c^3 \right),\] where $\alpha \in K$ satisfies $\alpha^2+\alpha + 1 =0$ and $c$ is a root of $y^{10}+ay^5+a^2+b$. So $P(x)$ decomposes into product of two cubic polynomials after a degree $10$ extension. If $A_6 \subset G$, then the minimal degree of the extension over which $P(x)$ splits into a product of two cubic polynomials is \[\frac{\# S_6}{\# (S_3 \times S_3)}= \frac{\# A_6}{\# ((S_3 \times S_3) \bigcap A_6)}=20.\] Therefore $G \not\supset A_6$. The cubic polynomials appearing in the factorization are conjugate to each other over the field $K(a,b)(c)$ via an automorphism that fixes $c$ and sends $\alpha$ to $1 + \alpha$. Therefore the total degree of the extension is at most $10\#S_3 = 60$. Thus $G=\PSL_2(5)$.  
	
	\item $n=12$.  If $p=3$, then $G=M_{11}$, see \cite{Uchida1970}*{Example~3}. Assume $p=2$. In this case either $G$ is one of the Mathieu groups $M_{11}, M_{12}$ or $G$ contains $A_n$. Let $\widehat{G}$ denote the Galois group of $P(x)$ over $\F_2(a,b)$. Then $G$ is a normal subgroup of $\widehat{G}$. Over $\F_2$ the polynomial $x^{12}+x+1$ factorizes into a product of irreducibles of degrees $3,4$ and $5$. Therefore $\widehat{G}$ is a $2$-transitive group containing a $4$-cycle and thus $\widehat{G}=S_{12}$. Since $G$ is a normal subgroup of $\widehat{G}$, the group $G$ contains $A_{12}$.
	
	\item $n=24$. In this case either $G=M_{24}$ or $G \supset A_{24}$. If $p=2$, then we can apply Serre's linearization method \cite{Abhyankar1993}. An almost identical computation, but for the trinomial $x^{24}+x+t$ over the field $K(t)$ is in \cite{Conway-et-al2010} We will find an additive polynomial that is divisible by $x^{24}+ax+b$. Consider the following sequence of equalities in $K(a,b)[x]/(x^{24}+ax+b)$:
	\begin{multline*}
	0=x^{32}+ax^9+bx^8=x^{256}+a^8 x^{72}+b^8 x^{64}\\ = x^{256}+a^8 (ax+b)^3+
	b^8 x^{64} 	=x^{256}+b^8x^{64}+a^{11}x^3+a^{10}bx^2+a^9b^2x+a^8b^3\\
	= x^{2048}+b^{64}x^{512}+a^{88}x^{24}+a^{80}b^8x^{16} + a^{72}b^{16}x^8+a^{64}b^{24}\\
	 = x^{2048}+b^{64}x^{512}+a^{80}b^8x^{16} + a^{72}b^{16}x^8+a^{89}x + a^{88}b+a^{64}b^{24}
	\end{multline*}
The last polynomial in this chain of equalities is an additive polynomial up to a constant. Therefore the group $G$ is a subquotient of $\AGL_{11}(2)$, thus $G$ cannot contain $A_{24}$ since $23 \nmid \# \AGL_{11}(2)$. Hence $G=M_{24}$.

Suppose $p=3$. Let $\widehat{G}$ denote the Galois group of $P(x)$ over $\F_9(a,b)$. Then $\widehat{G}$ contains $G$ as a normal subgroups, and therefore $\widehat{G}$ is also a primitive group containing an $n$-cycle. From Theorem~\ref{Jones} it follows that $\widehat{G}$ is either one of the groups $\PSL_2(23), \PGL_2(23), M_{24}$  or $\widehat{G} \supset A_{24}$. Let $c \in \F_9$ be a root of the polynomial $y^2-y-1$. The polynomial $x^{24}-x+c$ factors over $\F_9$ into a product of irreducible polynomials of degrees $1,2,3,4,6$ and $8$. Therefore $\widehat{G}$ contains an element $g$ with cycle type $1,2,3,4,6,8$. Since $g^8$ fixes at least $15$ points, $\widehat{G} \not\subset \PGL_2(23)$. The group $M_{24}$ has no elements with the cycle type of $g$ (see Appendix~\ref{MathieuCycles}). Therefore $\widehat{G} \supset A_n$. Since $G$ is a normal subgroup of $\widehat{G}$, the group $G$ contains $A_{24}$. 
\end{enumerate}

\item Assume $p|m$. Consider the specialization of $x^n+ax^m+b$ to the polynomial $P(x)=x^n+t$ over $\Klaur{t}$. The Galois group of this polynomial is cyclic and we deduce that $G$ contains an $n$-cycle. Theorem~\ref{Jones} gives a finite list of possibilities for $G$. Consider the specialization to the polynomial $Q(x)=x^n+tx^m+t^2$ over $\Klaur{t}$. The Newton polygon of this polynomials has two slopes $1/(n-m)$ and $1/m$. Since $p|m$ the roots of $Q$ span a wildly ramified extension. Any element of the wild inertia will fix at least $n-m > n/2$ points. No group from the case~\ref{cycle0} of Theorem~\ref{Jones} can have nontrivial elements fixing at least half of the points. 

\item Assume $p|n-m$. Consider the specialization $x^n+t$ over $\Klaur{t}$. Since $n,m$ are relatively prime, the specialized polynomial is separable with cyclic Galois group. Assume $A_n \not\subset G$. Since $G$ contains an $n$-cycle, Theorem~\ref{Jones} produces a list of possibilities for $G$. We analyze each case separately.

\begin{enumerate}
	\item\label{4a} $C_\ell \subset G \subset \AGL_1(\ell)$, $n=\ell\geqslant 5$ is prime. Consider the specialization $x^n+t^{-1}x^m + 1$ over $\Klaur{t}$. The Newton polygon of this trinomial consists of two segments with slopes $m$ and $m-n$. Since $p|(n-m)$ the extension is wildly ramified. Any element from the wild inertia subgroup fixes at least $m$ roots. Therefore $m=1$. In this case Lemma~\ref{trinomial-inertia} applies with $\#B=1$. The stabilizer of $B$ in $G \subset \AGL_1(\ell)$ is abelian and has $\AGL_1(q)$ as a subquotient, thus $q=2$. Since $\ell \geqslant 5$ the group $H_1$ has an element $g \neq id$. The element $g$ fixes at least two points, contradiction.
	
	\item\label{PGL} $\PGL_r(\ell) \subset G \subset \PGaL_r(\ell)$,  $n=\# \PP^{r-1}(\F_\ell)=(\ell^r-1)/(\ell-1)$. We will prove that $\ell$ is necessarily a power of $p$ and that $m=\# \PP^{r'}(\F_{\ell})$ for some $r'<r$.

	  We can assume that $\PGL_r(\ell) \not\supset A_n$. We call a point blue if it is in the set $B$ and azure if it is in one of the $A_i$'s. Lemma~\ref{trinomial-inertia} implies that for every pair of points $P, Q \in A_i$ there is an element of $G$ that fixes all blue points and moves $P$ to $Q$. Our goal is to show that $p | \ell$ and $B$ forms a projective subspace of $\Omega$. To this end we examine the Galois action on lines that contain blue points. 
	 	
		Consider a line $l$ in $\PP^r(\F_\ell)$ that contains at least $3$ blue points; identify this line with $\PP^1(\F_\ell)$. Any element of $G$ that fixes all blue points is a power of the Frobenius when restricted to $l$.  Since any azure point can be moved by an element of $G$ that fixes all blue points, the set of blue points of $l$ is the set $l(\F_{t})$ for some subfield $\F_t \subset \F_\ell$. There is a subgroup of $H$ that fixes all blue points and acts transitively on $A_i$ for every $i$. Thus if $l$ contains a point from $A_i$, then it contains the whole set $A_i$. Assume $A_i, A_j \subset l$ for some $i \neq j$. The group $H_i$ fixes all blue points and $A_i$, and acts transitively on the rest of the line. Therefore $A_i$ equals $l(\F_{t^k}) \setminus l(\F_t)$, where $k=\log_t(q+t)$. Similarly $A_j$ equals $l(\F_{t^k}) \setminus l(\F_t)$, contradiction. Thus $l$ contains points only from one group $A_i$. Since the subgroup of $G$ that fixes $B$ acts transitively on $A_i$, the extension $\F_\ell/\F_t$ is such that the Frobenius acts with only one orbit. There is only one such extension $\F_4/\F_2$. So either $\ell=4, q=2$ or every line containing at least three blue points consists entirely of blue points. Suppose $\ell = 4, q=2$. If $r=2$, then $\PGL_2(4)=A_5$, contradiction. Therefore $r>2$. Suppose $m > 3$. Take a blue  point $Q$ not on $l$ and let $P\in l$ be an azure point $P \in A_i$. Then $A_i \subset l$. Consider the line connecting $Q$ and $P$. If this line contains a blue point $R \neq Q$, then an element of $H$ that induces a transposition on $A_i$ does not preserve the collinearity of $P,Q$ and $R$. If this line contains an azure point $R \in A_j$, then an element of $H_j$ that induces a transposition on $A_i$ does not preserve the collinearity of $P,Q$ and $R$, contradiction. Suppose $m=3, \ell=4, q=2, r>2$. Consider blue points $P, Q\in l$ and an azure point $R_1 \in A_i$, $R_1 \not\in l$ (as in Figure~\ref{threebluedots}). Let $R_2$ denote the point of $A_i \setminus \{R_1\}$. The  line $l'$ connecting $P$ and $R_1$ does not contain any blue points except $P$; therefore there is a point $S \in l'$, $S \in A_j$ for some $j \neq i$. Let $g \in H_j$ be an element that maps $R_1$ to $R_2$. Since $g$ fixes $S$ and $P$, $g$ preserves $l'$ and therefore $R_2 \in l'$ and $P,R_1, R_2$ are collinear. Similarly $Q, R_1$ and $ R_2$ are collinear, contradiction. Therefore \slantsf{every line containing at least three blue points consists entirely of blue points}. 
		\begin{figure}[h]
		\begin{minipage}{.5\textwidth}
			\centering
			\includegraphics[scale=0.8]{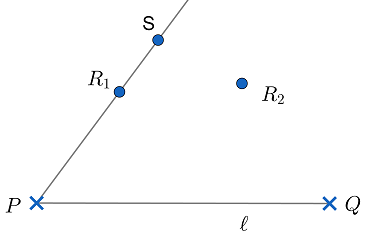}
			\caption{}
			\label{threebluedots}
		\end{minipage}%
		\begin{minipage}{.5\textwidth}
			\centering
			\includegraphics[scale=0.9]{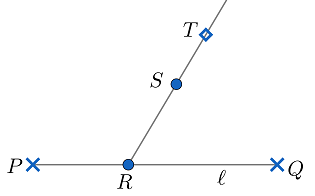}
			\caption{}
			\label{twobluedots}
		\end{minipage}
		
		\end{figure}

		 Consider a line $l$ containing exactly two blue points $P$ and $Q$ (if it exists). Consider a point $R \in A_i \cap l$. Since the pointwise stabilizer of $B$ in $G$ acts transitively on $A_i$, the whole set $A_i$ lies on $l$. We conclude that $q|\ell-1$. Suppose $m>2$, so there is a blue point $S \not\in l$ (as in Figure~\ref{twobluedots}). Let $l'$ denote the line connecting $S$ and $R$. Since $A_i \subset l$, there is a point $T \in (A_j \cup B) \cap l'$ for some $j \neq i$. Take an element $g \in H_j$ that moves $R$ to a different point on $l$. Then $g$ fixes $S$ and $T$, and therefore fixes $l'$ but moves $R$ to a point not on $l'$, contradiction. If $m=2$ and $r>2$, there exists a line containing exactly one blue point.
		 Therefore, assuming that $l$ exists, $q| \ell - 1$ and either $r=2$ or there exists a line containing exactly one blue point.
	 
		 Assume that there exists a line $l$ that passes through exactly one blue point. If all azure points on this line are from the same group $A_i$, take two azure points $P,Q$ on the line and consider an element of $G$ that maps $P$ to $Q$ and induces a translation on $\AGL_1(\F_q)$. This element preserves $l$ and decomposes all azure points into cycles of length $p$, hence $p|\ell$. If there is a point $P \in l$, $P \in A_i$ and a point $Q \in l \cap A_j, j \neq i$, consider elements of $H_i$ that fix $P$ and maps $Q$ to a different point of $A_j$. Such elements preserve $l$. We deduce that for every $j=1,...,d$ if there is one point from $A_j$ on the line, the whole set $A_j$ is contained in the line. In this case we again get that $q | \ell$. 
		 
		We have proved the following statements.
		\begin{itemize}
			\item Every line that contains at least three blue point is entirely blue.
			\item If a line containing exactly two blue points exists, then $q|\ell-1$ and either $r=m=2$ or a line containing exactly one blue point exists.
			\item If a line containing exactly one blue point exists, then $p|\ell$. 
		\end{itemize}
		 Assume that it is not the case that $r=m=2$. Then if a line containing two blue points exists, then $q | \ell -1$ and a different line passes through exactly one blue point, which in turn implies $p| \ell$, contradiction. Therefore there are no lines that contain exactly two blue points. So every line that contains two blue points, contains at least three, thus is entirely blue. Take a line through a blue and an azure point. It must have exactly one blue point. Therefore $p|\ell$ and the blue points form a projective subspace.
		 
		 Now we show that the case $r=m=2$ is impossible. Suppose $r=m=2$.  Then $n=\ell+1$ must be odd, therefore $\ell=2^k$. Since $n=qd+2$, both $q$ and $d$ are odd. Suppose $d=1$. If $q \neq 3$, then $h^2$ has exactly three fixed points, so $h^2$ is the degree $2$ Frobenius. A nontrivial orbit of the degree $2$ Frobenius on $\Aff^1(\F_{2^k})$ has size at most $k$. Since $h^2$ has two nontrivial orbits and two fixed points $2+2k \geqslant 2^k$. We conclude that $k \leqslant 3$. Since the Frobenius must have two nontrivial orbits, $k$ is equal to $3$. Then $q=6$ which is absurd. If $d=1$ and $q=3$, then $G \subset \PGL_2(4) = A_5$, contradiction. Thus we can assume that $d\neq 1$ and, since $d$ is odd, $d \geqslant 3$.  In this case the element $h$ has exactly $d$ fixed points, thus $h$ is a power of the Frobenius, so $d-1$ is a power of $2$. Suppose $q \neq 3$. Since $h$ has a unique orbit of length $2$, $h$ is the degree $2$ Frobenius and the number of fixed points of $h$ is $3=\#\PP^1(\F_2)=d$. The group $H_1$ fixes $q+2 >3$ points, therefore $H_1$ is generated by a power of the Frobenius and $q+1$, which is the number of fixed points of $H_1$ on $\Aff^1$, is a power of $2$. Therefore $q \equiv -1 \pmod 4$; hence $2^k=\ell=qd+1=3q+1 \equiv 2 \pmod 4$. Thus $\ell=2, n=3$. But if $\ell=2$, then $G \supset \PGL_2(2)=S_3$, contradiction. Finally we have to consider the case $q=3$. Write $d=2^{g}+1$, then $3(2^g+1)=qd=n-2=2^k-1$. The last equation can be satisfied modulo $8$ only if $d=5, n=17$. The group $H_1$ has $q+m=5$ fixed points, and is therefore generated by the degree $4$ Frobenius. The degree $4$ Frobenius on $\PP^1(\F_{16})$ acts with orbits of size $1$ and $2$, so it cannot act transitively on $A_2$. Contradiction.

		Thus we proved that $\ell$ is a power of $p$, $n=1+\ell + ... + \ell^{r-1}$ and $m=1+ \ell + ... + \ell^{s-1}$.
		
	\item	Assume $n=1+\ell + ... + \ell^{r-1}$ and  $m=1+ \ell + ... + \ell^{s-1}$, for some power $\ell$ of the characteristic. We claim that for such $n,m$ the group $G$ is $\PGL_r(\ell)$. Consider the equation $x^n+ax^m+b$. Let $y=x^{\ell-1}$. The equation becomes $y^{\ell^r-1}+ay^{\ell^s-1}+b$. Multiplying by $y$ gives $y^{\ell^r} + ay^{\ell^s}+by$. The roots of the latter form an $\F_\ell$ vector space, so its Galois group is a subgroup of $\GL_r(\ell)$. Since the roots of $x^n+ax^m+b$ correspond to lines in the space of roots of $y^{\ell^r} + ay^{\ell^s}+by$, $G$ is contained in $\PGL_r(\ell)$. Since $G$ contains an $n$-cycle and $n \neq 11, 23$, Theorem~\ref{Jones} implies that either $\PGL_s(t) \subset G \subset \PGaL_s(t)$ for some prime power $t$ or $n$ is prime and $C_n \subset G \subset \AGL_1(n)$. From the case~\ref{4a} of this proof we deduce that $\PGL_s(t) \subset G \subset \PGaL_s(t)$ . From the case~\ref{PGL} we deduce that $t$ is the largest power of the characteristic that divides $n-1$; hence $t=\ell$. Therefore $G=\PGL_d(\ell)$.

	\item $n=11$ and $G=M_{11}$ or $G=\PSL_2(11)$. In this case $m$ can be any number from $1$ to $5$. Assume that $\widehat{G}$ is the Galois group of $P$ over some $F(a,b)$ for some finite field $F$. The group $G$ is a normal subgroup of $\widehat{G}$. Since $\widehat{G}$ contains an $n$-cycle, it is equal to one of $M_{11}, \PSL_2(11), A_{11}, S_{11}$.  Hence to show, that $A_{11} \subset G$ it suffices to show that $\widehat{G}\neq M_{11}, \PSL_2(11)$. To show that $\widehat{G} \neq M_{11}, \PSL_2(11)$ it is enough to find s specialization of $P(x)$ over a finite field, whose factorization pattern contradicts the possible cycle types of $M_{11}$ and $\PSL_2(11)$.  We do so in Table~\ref{factortable} of Appendix~\ref{Factorizations} (in every case the factorization shows that $\widehat{G}$ has a cycle of length $2,3$ or $5$ contradicting Theorem~\ref{Jones}).

	 The only case remaining is that of the polynomial $x^{11}+ax^2+b$ when $p=3$. But in this case Uchida \cite{Uchida1970}*{Example~4} proved that $G=M_{11}$.

	\item $n=23$ and $G=M_{23}$. Similarly to the previous case, for all but one possible pair $(p,m)$ we can find a trinomial over a finite field with factorization pattern  that is impossible for $G=M_{23}$. The results are summarized in Table~\ref{factortable} of Appendix~\ref{Factorizations}. The only case not covered in Table~\ref{factortable} is that of the polynomial $x^{23}+ax^3+b$ for $p=2$. The polynomial $x^{23}+tx^3+1$ over $\F_2(t)$ has Galois group $M_{23}$, as proved by Abhyankar using the linearization method \cite{Abhyankar1993}. Consider the field $L=K(a,b)(b^{1/23})$. Let $\alpha\colonequals b^{1/23}$ and $\beta\colonequals a\alpha^{-20}$. Then $L$ is isomorphic to the field of rational functions $K(\alpha, \beta)$. Over $L$ the equation $x^{23}+ax^3+b=0$ can be simplified. Let $x=\alpha y$, then the equation becomes $y^{24}+\beta y^3 + 1=0$. Since $K(\alpha, \beta)/K(\alpha)$ is a purely transcendental extension, the Galois group of $y^{23}+\beta y^3+1$ over $K(\alpha, \beta)$ is equal to the Galois group of the same equation over $K(\beta)$, which is equal to $M_{23}$. Since $L/K(a,b)$ is a cyclic extension, $G \not\supset A_{23}$, thus $G=M_{23}$.

\end{enumerate} 
\end{enumerate}

\end{proof}

\begin{remark}
	Theorem~\ref{final-trinomial} does not distinguish between the cases $G=A_n$ and $G=S_n$, but it is possible to do so. When $p$ is odd this is discussed in Remark~\ref{RathmannAn}. If $p=2$ and $n$ is even, then $G\subset A_n$ by \cite{Abhyankar-Oun-Avinash1994}*{Proposition~2.24}. If $p=2$ and $n$ is odd, then $G\subset A_n$ if and only if $m=2$ by \cite{Abhyankar-Oun-Avinash1994}*{Proposition~2.23}.
\end{remark}

\section*{Acknowledgements} 

I would like to thank my advisor Bjorn Poonen for careful reading of the paper and insightful conversations. I am grateful to Gareth Jones and an anonymous referee for pointing out an omission in the list of triply transitive groups in an earlier version of the paper. I thank Dmitri Kubrak for many helpful suggestions. 
\appendix

\section{Cycle types of Mathieu groups}\label{MathieuCycles}
The following table describes cycle types of Mathieu groups in their standard multiply transitive actions; see \cite{ATLASv3}.
	\begin{table}[h]
	\begin{center}
		\begin{tabularx}{\textwidth}{l l Y} \toprule
			Group & Number of points & Cycle types  \\ \midrule
			$M_{11}$ & $11$ & $2^4$, $3^3$, $4^2$, $5^2$, $(2,3,6)$, $(2,8)$, $11$ \\ \addlinespace
			$M_{11}$ & $12$ & $2^4$, $3^3$, $(2^2, 4^2)$, $5^2$, $(2,3,6)$, $(4,8)$, $11$ \\	\addlinespace
			$M_{12}$ & $12$ &
			$2^6$, $2^4$, $3^3$, $3^4$, $(2^2, 4^2)$,
			 $4^2$, $5^2$, $6^2$, 			  
			$(2,3,6)$, $(4,8)$, $(2,8)$, $(2,10)$, $11$ \\
			\addlinespace
			$M_{22}$ & $22$ & $2^8$, $3^6$,  $(2^2, 4^4)$, $5^4$, $(2^2, 3^2, 6^2)$, $7^3$,
			$(2,4,8^2)$, $11^2$ \\
				\addlinespace
			$\Aut(M_{22})$ & $22$ & $2^7$, $2^8$, $2^{11}$, $(2, 4^4)$, $(2^3, 4^4)$, $3^6$,  $(2^2, 4^4)$, $5^4$, $(2, 3^2, 6^2)$, $(2^2, 3^2, 6^2)$,  $7^3$,
			$(2,4,8^2)$, $(4, 8^2)$, $(2, 10^2)$, $11^2$, $(4, 6, 12)$, $(7, 14)$ \\
			\addlinespace
			$M_{23}$ & $23$ & $2^8$, $3^6$, $(2^2, 4^4)$, $5^4$, $(2^2, 3^2, 6^2)$, $7^3$, 			
			$(2,4,8^2)$, $11^2$, $(2,7,14)$, $(3,5,15)$, $23$ \\\addlinespace
			$M_{24}$ & $24$ & $2^8$, $2^{12}$, $3^6$, $3^8$, $(2^4, 4^4)$, $(2^2, 4^4)$, $4^6$, $5^4$, $(2^2, 3^2, 6^2)$, $6^4$, $7^3$,
			$(2,4,8^2)$, $2^2, 10^2$, $11^2$, $(2,4,6,12)$, $12^2$, $(2,7,14)$, $(3,5,15)$, $(2,21)$, $23$ \\

			\bottomrule
			
		\end{tabularx}\caption{Cycle types of Mathieu groups}\label{MathCycl}
	\end{center}
\end{table}
\clearpage
\section{Factors of trinomials over finite fields}\label{Factorizations}
	\begin{table}[h]
	\begin{center}
		\begin{tabularx}{275pt}{l l l} \toprule
			Polynomial & Field & Degrees of irreducible factors  \\ \midrule
			$x^{11}+x+1$ & $\F_2$ & $2 \times 9$ \\ 
			$x^{11}-x-1$ & $\F_5$ & $1 \times 3 \times 7$ \\	
			$x^{11}+x^3+1$ & $\F_2$ & $5 \times 6$\\	
			$x^{11}+x^{4}+1$ & $\F_7$ & $1\times1\times2\times7$ \\
			$x^{11}+x^{5}+1$ & $\F_3$ & $1\times3\times7$ \\
			$x^{11}+x^{5}+1$ & $\F_2$ & $3\times8$ \\
			\addlinespace
			$x^{23}+x+1$ & $\F_2$ & $2 \times 8\times 13$  \\ 
			$x^{23}+x+1$ & $\F_{11}$ & $1 \times 2 \times 5 \times 15$  \\	
			$x^{23}+x^2+1$ & $\F_3$ & $1 \times 2 \times 20$  \\	
			$x^{23}+x^{2}+1$ & $\F_7$ & $7\times 16$  \\
			$x^{23}+x^{3}+1$ & $\F_5$ & $1\times 22$  \\
			$x^{23}+x^{4}+1$ & $\F_{19}$ & $1\times 1 \times 1 \times 4 \times 7 \times 9$  \\
			$x^{23}+x^{5}+1$ & $\F_3$ & $1 \times 2\times 5 \times 7 \times 8$  \\			
			$x^{23}+cx^{5}+1$ & $\F_4$ & $1 \times 9 \times 13$  \\
			$x^{23}+x^{6}+1$ & $\F_{17}$ & $2\times 3 \times 9 \times 9$  \\
			$x^{23}+x^{7}+1$ & $\F_2$ & $2\times 10 \times 11$ \\
			$x^{23}+x^{8}+1$ & $\F_3$ & $1\times 3 \times 19$  \\
			$x^{23}+x^{8}+1$ & $\F_5$ & $1\times 4 \times 5 \times 13$ \\
			$x^{23}+x^{9}+c$ & $\F_4$ & $1\times 2 \times 20$ \\
			$x^{23}+x^{9}+1$ & $\F_7$ & $4\times 19$  \\
			$x^{23}+x^{10}+1$ & $\F_{13}$ & $1\times 1 \times 4 \times 6 \times 11$  \\
			$x^{23}+x^{11}+1$ & $\F_2$ & $5\times 6 \times 12$  \\
			$x^{23}+x^{11}+1$ & $\F_3$ & $1\times 3 \times 19$ \\

			\bottomrule
			
		\end{tabularx}\caption{Factorization of trinomials; $c^2=c+1$}\label{factortable}
	\end{center}
	
\end{table}

\end{document}